\documentclass[12pt]{article}

\usepackage{amsmath}
\usepackage{amsfonts}
\usepackage{amssymb}
\usepackage{amsthm}
\usepackage{graphicx}
\usepackage{cite}
\usepackage{caption}
\usepackage{url}
\usepackage{color}

\usepackage{algorithm}
\usepackage{algpseudocode}
\usepackage{mathtools}
\usepackage{subcaption}
\usepackage{multirow}
\usepackage{bbm}
\usepackage[table]{xcolor}

\newtheorem{lemma}{Lemma}
\newtheorem{theorem}{Theorem}
\newtheorem{corollary}{Corollary}

\newcommand{\norm}[1]{\left\lVert#1\right\rVert}
\newcommand\Tstrut{\rule{0ex}{2.7ex}}         
\newcommand\Bstrut{\rule[-1.3ex]{0pt}{0pt}}   
\newcommand\Strut{\rule{0ex}{0ex}}         

\definecolor{MyGreen}{rgb}{0,0,0}
\definecolor{red}{rgb}{0,0,0}

\title{``Interpolated Factored Green Function'' Method\\ for accelerated solution of Scattering Problems} \author{Christoph
 Bauinger\footnote{Computing and Mathematical Sciences, Caltech, Pasadena, CA 91125, USA} \and Oscar P. Bruno$^*$}
\begin{document}
\date{}
\maketitle
\begin{abstract}
  This paper presents a novel {\em Interpolated Factored Green
    Function} method (IFGF) for the accelerated evaluation of the
  integral operators in scattering theory and other areas. Like
  existing acceleration methods in these fields, the IFGF algorithm
  evaluates the action of Green function-based integral operators at a
  cost of $\mathcal{O}(N\log N)$ operations for an $N$-point surface
  mesh. The IFGF strategy, which leads to an extremely simple
  algorithm, capitalizes on slow variations inherent in a certain
  Green function {\em analytic factor}, which is analytic up to and
  including infinity, and which therefore allows for accelerated
  evaluation of fields produced by groups of sources on the basis of a
  recursive application of classical interpolation methods. Unlike
  other approaches, the IFGF method does not utilize the Fast Fourier
  Transform (FFT), and is thus better suited than other methods for
  efficient parallelization in distributed-memory computer
  systems. Only a serial implementation of the algorithm is considered
  in this paper, however, whose efficiency in terms of memory and
  speed is illustrated by means of a variety of numerical
  experiments---including a 43 min., single-core operator evaluation
  (on 10 GB of peak memory), with a relative error of
  $1.5\times 10^{-2}$, for a problem of acoustic size of 512$\lambda$.
\end{abstract}
\vspace{0.5 cm}
\maketitle
\newpage
\section{\label{sec:introduction}Introduction}

This paper presents a new methodology for the accelerated evaluation
of the integral operators in scattering theory and other areas. Like
existing acceleration methods, the proposed {\em Interpolated Factored
  Green Function} approach (IFGF) can evaluate the action of Green
function based integral operators at a cost of $\mathcal{O}(N\log N)$
operations for an $N$-point surface mesh. Importantly, the proposed
method does not utilize previously-employed acceleration elements such
as the Fast Fourier transform (FFT), special-function expansions,
high-dimensional linear-algebra factorizations, translation operators,
equivalent sources, or parabolic
scaling~\cite{2007DirectionalFMMLexing,2014ButterflyLexing,1996ButterflyMichielssen,ROKHLIN1993,2006FMMRokhlin,2017Boerm,2017BoermH2,2001FFTKunyansky,1996AIMFFTBelszynski,
  1997FFTPhillips,2003FMMLexingKernelIndependent,2003BebendorfRjasanow}. Instead,
the IFGF method relies on straightforward interpolation of the
operator kernels---or, more precisely, of certain factored forms of
the kernels---, which, when collectively applied to larger and larger
groups of Green function sources, in a recursive fashion, gives rise
to the desired $\mathcal{O}(N\log N)$ accelerated evaluation. The IFGF
computing cost is competitive with that of other approaches, and, in a
notable advantage, the method runs on a minimal memory footprint.  For
example, as shown in Table~\ref{table:timings5} below, a 43-minute,
single-core run on a mere 10 GB of peak memory suffice to produce the
full discrete operator evaluation, with a relative error of
$1.5\times 10^{-2}$, for a problem 512 wavelengths in acoustic size.
In sharp contrast to other algorithms, finally, the IFGF method is
extremely simple, and it lends itself to straightforward
implementations and effective parallelization.

As alluded to above, the IFGF strategy is based on the interpolation
properties of a certain factored form of the scattering Green
function into a singular and rapidly-oscillatory {\em centered factor}
and a slowly-oscillatory {\em analytic factor}. Importantly, the
analytic factor is analytic up to and including infinity (which
enables interpolation over certain unbounded conical domains on the
basis of a finite number of radial interpolations nodes), and, when
utilized for interpolation of fields with sources contained within a
cubic box $B$ of side $H$, it enables uniform approximability over
semi-infinite cones, with apertures proportional to $1/H$. In
particular, unlike the FMM based approaches, the algorithm does not
require separate treatment of the low- and high-frequency regimes. On
the basis of these properties, the IFGF method orchestrates the
accelerated operator evaluation utilizing two separate tree-like
hierarchies which are combined in a single boxes-and-cones hierarchical data structure. Thus, starting from an initial cubic box
of side $H_1$ which contains all source and observation points
considered, the algorithm utilizes, like other approaches, the octree
$\mathcal{B}$ of boxes that is obtained by partitioning the initial
box into eight identical child boxes of side $H_2 = H_1/2$ and iteratively
repeating the process with each resulting child box until the
resulting boxes are sufficiently small.

Along with the octree of boxes, the IFGF algorithm incorporates a
hierarchy $\mathcal{C}$ of {\em cone segments}, which are used to
enact the required interpolation procedures. Each box in the tree
$\mathcal{B}$ is thus endowed with a set of box-centered cone segments
at a corresponding level of the cone hierarchy $\mathcal{C}$. In
detail, a set of box-centered cone segments of extent $\Delta_{s,d}$
in the analytic radial variable $s$, and angular apertures
$\Delta_{\theta,d}$ and $\Delta_{\varphi,d}$ in each of the two
spherical angular coordinates $\theta$ and $\varphi$, are used for
each $d$-level box $B$. (Roughly speaking, $\Delta_{s,d}$,
$\Delta_{\theta,d}$ and $\Delta_{\varphi,d}$ vary in an inversely
proportional manner with the box size $H_d$ for large enough boxes,
but they remain constant for small boxes; full details are presented
in Section~\ref{sec:algnotation}.)  The set of cone segments centered
at a box $B$ is used by the IFGF algorithm to set up an interpolation
scheme over all of space around $B$, except for the region occupied by
the union of $B$ itself and all of its nearest neighboring boxes at
the same level. Thus, the leaves (level $D$) in the box tree, that is,
the cubes of the smallest size used, are endowed with cone segments of
largest angular and radial spans $\Delta_{s,D}$, $\Delta_{\theta,D}$
and $\Delta_{\varphi,D}$ considered. Each ascent $d\to (d-1)$ by one
level in the box tree $\mathcal{B}$ (leading to an increase by a
factor of two in the cube side $H_{d-1} = 2 H_d$) is accompanied by a
corresponding descent by one level (also $d\to (d-1)$) in the cone
hierarchy $\mathcal{C}$ (leading, e.g., for large boxes, to a decrease
by a factor of one-half in the radial and angular cone spans:
$\Delta_{s,{d-1}}=\frac 12\Delta_{s,d}$,
$\Delta_{\theta,{d-1}}=\frac 12\Delta_{\theta,d}$ and
$\Delta_{\varphi,{d-1}}=\frac 12\Delta_{\varphi,d}$; see
Section~\ref{sec:algnotation}). In view of the interpolation
properties of the analytic factor, the interpolation error and cost
per point resulting from this conical interpolation setup remains
unchanged from one level to the next as the box tree is traversed
towards its root level $d=1$. The situation is even more favorable in
the small-box case. And, owing to analyticity at infinity,
interpolation for arbitrarily far regions within each cone segment can
be achieved on the basis of a finite amount of interpolation data. In
all, this strategy reduces the computational cost, by commingling the
effect of large numbers of sources into a small number of
interpolation parameters. A recursive strategy, in which cone segment
interpolation data at level $d$ is also exploited to obtain the
corresponding cone-segment interpolation data at level $(d-1)$,
finally, yields the optimal $\mathcal{O}(N\log N)$ approach.
		 
The properties of the factored Green function, which underlie the
proposed IFGF algorithm, additionally provide certain perspectives
concerning various algorithmic components of other acceleration
approaches. In particular, the analyticity properties of the analytic factor, which are established in Theorem~\ref{theorem:Error}, in
conjunction with the classical polynomial interpolation bound
presented in Theorem~\ref{theorem:errorestimatenested}, and the IFGF
spherical-coordinate interpolation strategy, clearly imply the
property of low-rank approximability which underlies some of the ideas
associated with the butterfly~\cite{1996ButterflyMichielssen,
  2014ButterflyLexing, Cands2009FastButterfly} and directional FMM
methods~\cite{2007DirectionalFMMLexing}. The directional FMM approach,
further, relies on a ``directional factorization'' which, in the
context of the present interpolation-based viewpoint, can be
interpreted as facilitating interpolation. For the directional
factorization to produce beneficial effects it is necessary for the
differences of source and observation points to lie on a line
asymptotically parallel to the vector between the centers of the
source and target boxes. This requirement is satisfied in the
directional FMM approach through its ``parabolic-scaling'', according
to which the distance to the observation set is required to be the
square of the size of the source box. The IFGF factorization is not directional, however, and it does not require use of the parabolic
scaling: the IFGF approach interpolates analytic-factor contributions
at linearly-growing distances from the source box.

In a related context we mention the recently introduced
approach~\cite{2017BoermH2}, which incorporates in an
${\mathcal H}^2$-matrix setting some of the main ideas associated with
the directional FMM algorithm~\cite{2007DirectionalFMMLexing}. Like
the IFGF method, the approach  relies on interpolation of a
factored form of the Green function---but using the directional
factorization instead of the IFGF factorization. The method yields a
full LU decomposition of the discrete integral operator, but it does so
under significant computing costs and memory requirements, both for
pre-computation, and per individual solution.

It is also useful to compare the IFGF approach to other acceleration
methods from a purely algorithmic point of view. The FMM-based
approaches~\cite{2007DirectionalFMMLexing, 2017Boerm, 2006FMMRokhlin,
  2012FMMChebyshevMessnerSchanz} {\color{MyGreen} entail two passes over
  the three-dimensional acceleration tree, one in the upward
  direction, the other one downward. In the upward pass of the
  original FMM methods, for example, the algorithm commingles
  contributions from larger and larger numbers of sources via
  correspondingly growing spherical harmonics expansions, which are
  sequentially translated to certain spherical coordinate systems and
  then recombined, as the algorithm progresses up the tree via
  application of a sequence of so-called M2M translation operators
  (see e.g. \cite{2007DirectionalFMMLexing}). In the downward FMM
  pass, the algorithm then re-translates and localizes the
  spherical-harmonic expansions to smaller and smaller boxes via
  related M2L and L2L translation operators
  (e.g. \cite{2007DirectionalFMMLexing}). The algorithm is finally
  completed by evaluation of surface point values at the end of the
  downward pass. } The IFGF algorithm, in contrast, progresses
simultaneously along two tree-like structures, the box tree and the
cone interpolation hierarchy, and it produces evaluations at the
required observation points, via interpolation, at all stages of the
acceleration process (but only in a neighborhood of each source box at
each stage). In particular, the IFGF method does not utilize
high-order expansions of the kinds used in other acceleration
methods---and, thus, it avoids use of Fast Fourier Transforms (FFTs)
which are almost invariably utilized in the FMM to manipulate the
necessary spherical harmonics expansions.
(Reference~\cite[Sec. 7]{Gumerov2004} mentions two alternatives which,
however, it discards as less efficient than an FFT-based procedure.)
The use of FFTs presents significant challenges, however, in the
context of distributed memory parallel computer systems. In this
regard reference~\cite{2007DirectionalFMMLexing} (further referencing
~\cite{2003FMMLexingKernelIndependent}), for example, indicates ``the
top part of the [FMM] octree is a bottleneck'' for parallelization,
and notes that, in view of the required parabolic scaling, the
difficulty is not as marked for the directional FMM approach proposed
in that contribution. In~\cite{ParallelFMMChandramowlishwaran2010} the
part of the FMM relying on FFTs is identified to become a bottleneck
in the parallelization and it is stated that this difficulty occurs as
the FFT-based portion of the algorithm is subject to the ``lowest
arithmetic intensity'' and is therefore ``likely suffering from
bandwidth contention''.

The IFGF algorithm, which relies on interpolation by means of Chebyshev expansions of relatively low degree, does not require the
use of FFTs---a fact that, as suggested above, provides significant
benefits in the distributed memory context. As a counterpoint,
however, the low degree Chebyshev approximations used by the IFGF
method do not yield the spectral accuracy resulting from the
high-order expansions used by other methods. A version of the IFGF
method which enjoys spectral accuracy could be obtained simply by
replacing its use of low-order Chebyshev interpolation by Chebyshev
interpolation of higher and higher orders on cone segments of fixed
size as the hierarchies are traversed toward the root $d=1$. Such a
direct approach, however, entails a computing cost which increases
quadratically as the Chebyshev expansion order grows---thus degrading
the optimal complexity of the IFGF method. But the needed evaluation
of high-order Chebyshev expansions on arbitrary three-dimensional
grids can be performed by means of FFT-based interpolation methods
similar to those utilized in~\cite[Sec. 3.1]{2001FFTKunyansky}
and~\cite[Remark 7]{BRUNOLintner2013}. This approach, which is not
pursued in this paper, leads to a spectrally convergent version of the
method, which still runs on essentially linear computing time and
memory. But, as it reverts to use of FFTs, the strategy re-introduces
the aforementioned disadvantages concerning parallelization, which are
avoided in the proposed IFGF approach.

It is also relevant to contrast the algorithmic aspects in the IFGF
approach to those used in the butterfly
approaches~\cite{2014ButterflyLexing ,1996ButterflyMichielssen,
  Cands2009FastButterfly}. Unlike the interpolation-based IFGF, which
does not rely on use of linear-algebra factorizations, the butterfly
approaches are based on low rank factorizations of various
high-dimensional sub-matrices of the overall system matrix. Certain
recent versions of the butterfly methods reduce linear-algebra
computational cost by means of an interpolation process in
high-dimensional space in a process which can easily be justified on
the basis of the analytic properties of the factored Green function
described in Section~\ref{subsec:analyticity}. As in the IFGF
approach, further, the data structure inherent in the butterfly
approach~\cite{2014ButterflyLexing, 1996ButterflyMichielssen} is
organized on the basis of two separate tree structures that are
traversed in opposite directions, one ascending and the other
descending, as the algorithm progresses. In the
method~\cite{2014ButterflyLexing} the source and observation cubes are
paired in such a way that the product their sizes remains
constant---which evokes the IFGF's cone-and-box sizing condition,
according to which the angles scale inversely with the cone span
angles. These two selection criteria are indeed related, as the
interpolability by polynomials used in the IFGF approach has direct
implications on the rank of the interpolated values. But, in a
significant distinction, the IFGF method can be applied to a wide
range of scattering kernels, including the Maxwell, Helmholtz, Laplace
and elasticity kernels among others, and including smooth as well as
non-smooth kernels. The butterfly
approaches~\cite{Cands2009FastButterfly,2014ButterflyLexing}, in
contrast, only apply to Fourier integral operators with smooth
kernels. The earlier butterfly
contribution~\cite{1996ButterflyMichielssen} does apply to Maxwell
problems, but its accuracy, specially in the low-frequency
near-singular interaction regime, has not been studied in detail.


Whereas no discussion concerning parallel implementation of the IFGF
approach is presented in this paper, we note that, not relying on
FFTs, the approach is not subject to the challenging FFT communication
requirements inherent in all of the aforementioned
Maxwell/Helmholtz/Laplace algorithms. In fact,
experience in the case of the butterfly method~\cite{2014ButterflyLexing} for
non-singular kernels, whose data structure is, as mentioned above,
similar to the one utilized in the IFGF method, suggests that
efficient parallelization to large numbers of processors may hold for
the IFGF algorithm as well. In \cite{2014ButterflyLexing} this was achieved due to ``careful manipulation of bitwise-partitions of the product space of the source and target domains'' to ``keep the data (...) and the computation (...) evenly distributed''.

{\color{red}
This paper is organized as follows: after preliminaries are briefly
considered in Section~\ref{sec:preliminaries}, Section~\ref{sec:ifgf}
presents the details of the IFGF algorithm---including, in
Sections~\ref{subsec:analyticity} and~\ref{subsec:interpolation}, a
theoretical discussion of the analyticity and interpolation properties
of the analytic factor, and then, in Section~\ref{subsec:algorithm},
the algorithm itself. The numerical results presented in
Section~\ref{sec:examples} demonstrate the efficiency of the IFGF
algorithm in terms of memory and computing costs by means of several numerical experiments performed on different geometries with acoustic size up to 512 wavelengths. A few concluding
comments, finally, are presented in Section~\ref{sec:conclusions}.}


\section{Preliminaries and Notation} \label{sec:preliminaries} We consider discrete integral operators of the form
\begin{equation} \label{eq:field1}
    I(x_\ell) \coloneqq \sum \limits_{\substack{m = 1 \\ m \neq \ell}}^N a_m G(x_\ell, x_m) ,\quad \ell = 1, \ldots, N,
\end{equation}
on a two-dimensional surface $\Gamma\subset \mathbb{R}^3$, where $N$
denotes a given positive integer, and where, for $m = 1, \ldots, N$,
$x_m\in\Gamma$ and $a_m\in\mathbb{C}$ denote pairwise different points
and given complex numbers, respectively; the set of all $N$ surface
discretization points, in turn, is denoted by
$\Gamma_N \coloneqq \{x_1, \ldots, x_N\}$. For definiteness,
throughout this paper we focus mostly on the challenging
three-dimensional Helmholtz Green function case,
\begin{equation} \label{eq:greensfunction}
G(x, x')  = \frac{e^{\imath \kappa |x - x'|}}{4 \pi |x - x'|},
\end{equation}
where $\imath$, $\kappa$ and $|\cdot |$ denote the {\em imaginary unit}, the {\em wavenumber} and the {\em Euclidean norm} in $\mathbb{R}^3$, respectively. Discrete operators of the form~\eqref{eq:field1}, with various kernels $G$, play major roles in a wide range of areas in science and engineering, with applications to acoustic and electromagnetic scattering by surfaces and volumetric domains in two- and three-dimensional space, potential theory, fluid flow, etc. As illustrated in Section~\ref{subsec:interpolation} for the Laplace kernel
$G(x, x') = 1/|x - x'|$, the proposed acceleration methodology
applies, with minimal variations, to a wide range of smooth and
non-smooth kernels---including but not limited to, e.g. the Laplace, Stokes and
elasticity kernels, and even kernels of the form $G(x, x') = \exp{\left(\imath \varphi(x - x')\right)}$ for smooth functions $\varphi$. The restriction to surface problems, where the point sources lie on a two dimensional surface $\Gamma$ in three dimensional space, is similarly adopted for definiteness: the extension of the method to volumetric source distributions is straightforward and should prove equally effective.

Clearly, a direct evaluation of $I(x)$ for all $x \in \Gamma_N$
requires $\mathcal{O}(N^2)$ operations. This quadratic algorithmic
complexity makes a direct operator evaluation unfeasible for many
problems of practical interest. In order to accelerate the
evaluation, the proposed IFGF method partitions the surface points
$\Gamma_N$ by means of a hierarchical tree structure of {\color{MyGreen}{\em boxes}}, as
described in Section~\ref{subsec:algorithm}. The evaluation of the
operator \eqref{eq:field1} is then performed on basis of a small
number of pairwise box interactions, which may occur either
horizontally in the tree structure, between two nearby equi-sized
boxes, or vertically between a child box and a neighboring
parent-level box. As shown in Section~\ref{sec:ifgf}, the box
interactions can be significantly accelerated by means of a certain
interpolation strategy that is a centerpiece in the IFGF approach. The
aforementioned box tree, together with an associated cone structure,
are described in detail in Section~\ref{subsec:algorithm}.

To conclude this section we introduce the box, source-point and
target-point notations we use in what follows. To do this, for given
$H > 0$ and $x = ((x)_1, (x)_2, (x)_3)^T \in \mathbb{R}^3$ we define
the {\em axis aligned} box $B(x, H)$ of {\em box side} $H$ and
centered at $x$ as
\begin{equation} \label{eq:defbox} B(x, H) \coloneqq \left[ (x)_1 -
    \frac{H}{2}, (x)_1 +\frac{H}{2} \right) \times \left[ (x)_2 -
    \frac{H}{2}, (x)_2 +\frac{H}{2} \right) \times \left[ (x)_3 -
    \frac{H}{2}, (x)_3 +\frac{H}{2} \right);
\end{equation}
see Figure~\ref{fig:Boxes}. For a given {\color{MyGreen} {\em source box}} $B(x_S, H)$ of
side $H$ and centered at a given point
$x_S = \left( (x_S)_1, (x_S)_2, (x_S)_3 \right)^T \in \mathbb{R}^3$,
we use the enumeration
$x_1^S, \ldots, x_{N_S}^S \in B(x_S, H) \cap \Gamma_N$ ($N_S\leq N$
and, possibly, $N_S=0$) of all source points $x_m$, $m=1,\dots,N$,
which are contained in $B(x_S, H)$; the corresponding source
coefficients $a_m$ are denoted by $a_\ell^S \in \{a_1, \ldots, a_N\}$,
$\ell = 1, \ldots, N_S$. A given set of $N_T$ surface target points,
at arbitrary positions outside $B(x_S, H)$, are denoted by
$x_1^T, \ldots, x_{N_T}^T \in \Gamma_N \setminus B(x_S, H)$. Then,
letting $I_S = I_S(x)$ denote the field generated at a point $x$ by
all point sources contained in $B(x_S, H)$, we will consider, in
particular, the problem of evaluation of the local operator
\begin{equation} \label{eq:fieldboxes} I_S(x_\ell^T) \coloneqq \sum \limits_{m
    = 1}^{N_S} a_m^S G(x_\ell^T, x_m^S), \qquad \ell = 1, \ldots, N_T.
\end{equation}
A sketch of this setup is presented in Figure~\ref{fig:Boxes}.

\begin{figure}
    \centering
    \includegraphics[width=0.7\textwidth]{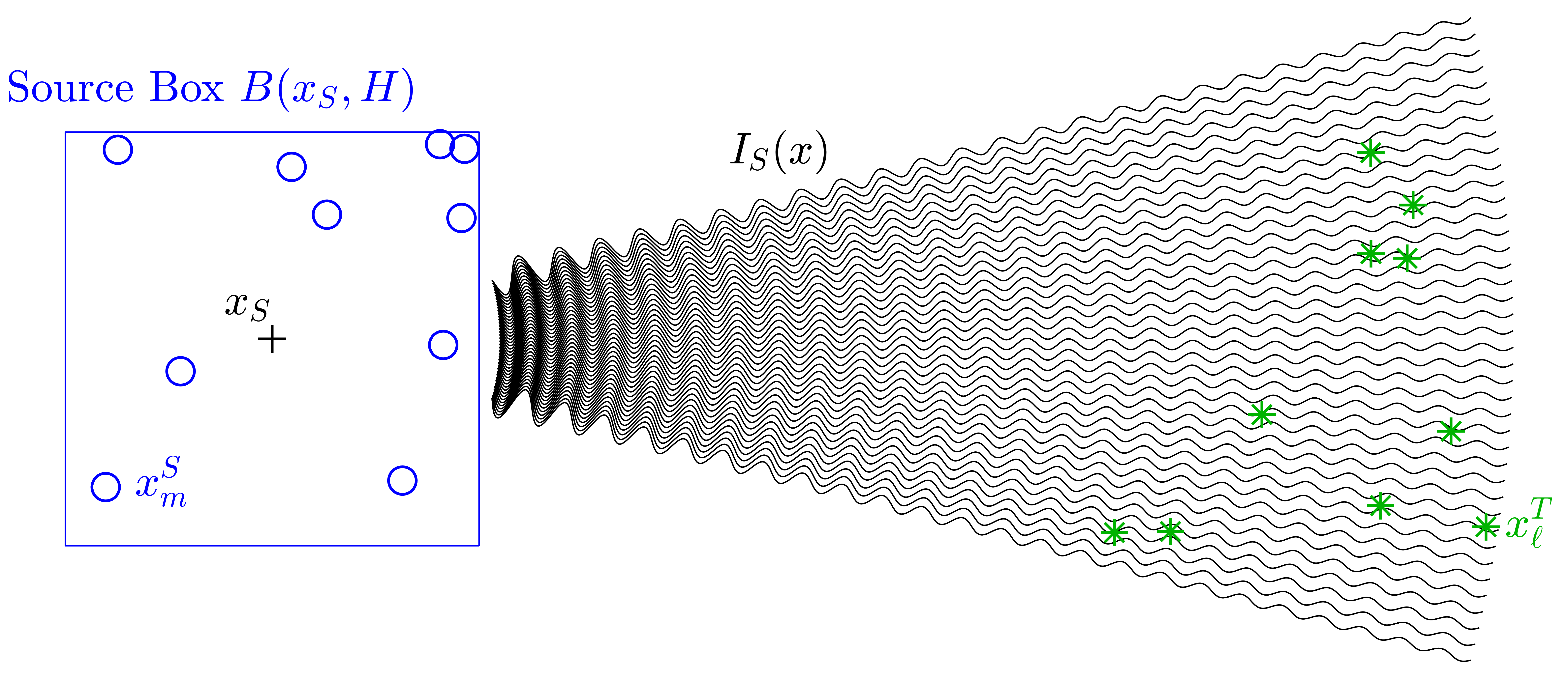}
    \caption{Two-dimensional illustration of a source box $B(x_S, H)$
      containing source points $x_1^S, x_2^S, x_3^S, \ldots$ (blue
      circles) and target points $x_1^T, x_2^T, x_3^T, \ldots$ (green
      stars). The black wavy lines represent the field $I_S$ generated
      by the point sources in $B(x_S, H)$.}
    \label{fig:Boxes}
  \end{figure}
  

\section{\label{sec:ifgf}The IFGF Method}
To achieve the desired acceleration of the discrete operator~\eqref{eq:field1}, the IFGF approach utilizes a certain factorization of the Green function $G$ which leads to efficient
evaluation of the field $I_S$ in equation~\eqref{eq:fieldboxes} by
means of numerical methods based on polynomial interpolation.

The IFGF factorization for $x'$ in the box $B(x_S, H)$ (centered at $x_S 
$) takes the form
\begin{equation}\label{eq:factor}
G(x, x') = G(x,x_S) g_S(x, x').
\end{equation}
Throughout this paper the functions $G(x,x_S)$ and $g_S$ are  called the {\em centered factor} and the {\em analytic factor},
respectively. Clearly, for a fixed given center $x_S$ the centered factor depends only on $x$: it is independent of $x'$. As shown in Section~\ref{subsec:analyticity}, in turn, the analytic factor is {\em analytic up to and including infinity} in the $x$ variable for each fixed value of $x'$ (which, in particular, makes $g_S(x, x')$ slowly oscillatory and asymptotically constant as a function of $x$ as $|x|\to\infty$), with oscillations as a function of $x$ that, for $x'\in B(x_S, H)$, increase linearly with the box size $H$.

Using the factorization~\eqref{eq:factor} the field $I_S$ generated by point sources placed within the source box $B(x_S, H)$ at any point
$x \in \mathbb{R}^3$ may be expressed in the form
\begin{equation}\label{eq:definitionF}
  I_S(x) = \sum \limits_{m = 1}^{N_S} a_m^S G(x, x_m^S) 
  = G(x, x_S)  F_S(x) \quad\mbox{where}\quad F_S(x) = \sum \limits_{m = 1}^{N_S} a_m^S g_S(x, x_m^S).
\end{equation}
The desired IFGF accelerated evaluation of the
operator~\eqref{eq:fieldboxes} is achieved via interpolation of the
function $F_S(x)$, which, as a linear combination of analytic factors,
is itself analytic at infinity. The singular and oscillatory character
of the function $F_S$, which determine the cost required for its
accurate interpolation, can be characterized in terms of the analytic
properties, mentioned above, of the factor $g_S$. A study of these
analytic and interpolation properties is presented in
Sections~\ref{subsec:analyticity} and \ref{subsec:interpolation}.

On the basis of the aforementioned analytic properties the algorithm evaluates all the sums in equation~\eqref{eq:fieldboxes} by first obtaining values of the function $F_S$ at a small number $P \in \mathbb{N}$ of points $p_i \in \mathbb{R}^3$, $i = 1, \ldots, P$, from which the necessary $I_S$ values (at all the target points $x_1^T, \ldots, x_{N_T}^T$) are rapidly and accurately obtained by interpolation. At a cost of $\mathcal{O}(P N_S + P N_T)$ operations, the interpolation-based algorithm yields useful acceleration provided $P \ll \min \{N_S, N_T\}$. Section \ref{subsec:algorithm} shows that adequate utilization of these elementary ideas leads to a multi-level algorithm which applies the forward map~\eqref{eq:field1}
for general surfaces at a total cost of $\mathcal{O}(N \log N)$ operations. The algorithm (which is very simple indeed) and a study of its computational cost are presented in
Section~\ref{subsec:algorithm}. 

In order to proceed with this program we introduce certain notations
and conventions. On one hand, for notational simplicity, but without
loss of generality, throughout the remainder of this section we assume
$x_S = 0$; the extension to the general $x_S \neq 0$ case is, of
course, straightforward. Incorporating the convention $x_S = 0$, then,
we additionally consider, for $0<\eta< 1$, the sets
\begin{equation*}
  A_\eta \coloneqq \{(x,x')\in\mathbb{R}^3\times \mathbb{R}^3\, : \, |x'| \leq \eta |x|\}
\end{equation*}
and
\begin{equation} \label{eq:defAdeltaH}
  A_{\eta}^H \coloneqq A_{\eta}\cap \left( \mathbb{R}^3\times B(x_S, H) \right).
\end{equation}
Clearly, $A_{\eta}^H$ is the subset of pairs in $A_\eta$ such that
$x'$ is restricted to a particular source box $B(x_S,
H)$. Theorem~\ref{theorem:Error} below implies that, on the basis of
an appropriate change of variables which adequately accounts for the
analyticity of the function $g_S$ up to and including infinity, this
function can be accurately evaluated for $(x,x')\in A_{\eta}^H$ by
means of a straightforward interpolation rule based on an
interpolation mesh in spherical coordinates which is very sparse along
the radial direction.

\subsection{\label{subsec:analyticity}Analyticity}

As indicated above, the analytic properties of the factor $g_S$ play a pivotal role in the proposed algorithm. Under the $x_S = 0$ convention established above, the factors in equation~\eqref{eq:factor} become
\begin{equation} \label{eq:factored_f} G(x, 0) =
  \frac{e^{\imath \kappa |x|}}{4 \pi |x|}\quad \mbox{and}\quad g_S(x, x') = \frac{|x|}{|x - x'|} e^{\imath \kappa \left( |x-x'| - |x|\right)}.
\end{equation}

In order to analyze the properties of the factor $g_S$ we introduce
the spherical coordinate parametrization 
\begin{equation} \label{eq:defparametrizationr} \tilde
  {\mathbf x}(r, \theta, \varphi) \coloneqq \begin{pmatrix} r\sin
    \theta \cos \varphi \\ r\sin \theta \sin \varphi \\ r\cos
    \theta\end{pmatrix}, \qquad 0\leq r< \infty, \mkern5mu 0\leq
  \theta \leq \pi, \mkern5mu 0\leq \varphi < 2\pi,
\end{equation}
and note that~\eqref{eq:factored_f} may be re-expressed in the form
\begin{equation}\label{factor_r}
  g_S(x, x') = \frac{1}{4 \pi \left | \frac{x}{r} - \frac{x'}{r}\right|}  \exp \left({\imath \kappa r\left( \left | \frac{x}{r}-\frac{x'}{r}\right| - 1\right)} \right).
\end{equation}
The effectiveness of the proposed factorization is illustrated in
Figures \ref{fig:radialsetup1_setup},
\ref{fig:radialsetup1_unfactored}, and
\ref{fig:radialsetup1_factored}, where the oscillatory character of
the analytic factor $g_S$ and the Green function
\eqref{eq:greensfunction} without factorization are compared, as a
function of $r$, for several wavenumbers. The slowly-oscillatory
character of the factor $g_S$, even for acoustically large source
boxes $B(x_S, H)$ as large as twenty wavelengths $\lambda$
($H = 20 \lambda$) and starting as close as just $3 H/2$ away from the
center of the source box, is clearly visible in Figure~\ref{fig:radialsetup1_factored}; much
faster oscillations are observed in Figure
\ref{fig:radialsetup1_unfactored}, even for source boxes as small as
two wavelengths in size ($H = 2 \lambda$). Only the real part is
depicted in Figures \ref{fig:radialsetup1_setup},
\ref{fig:radialsetup1_unfactored}, and \ref{fig:radialsetup1_factored}
but, clearly, the imaginary part displays the same behavior.
\begin{figure}
\centering
\begin{subfigure} {0.85\textwidth}
    \centering
    \includegraphics[width=0.9\textwidth]{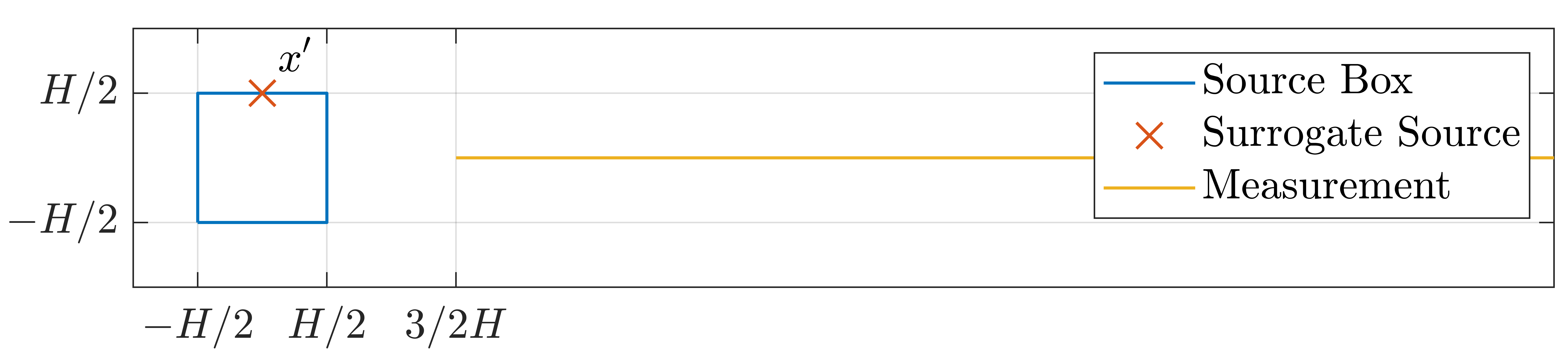}
    \caption{Test setup. The Surrogate Source position $x'$ gives rise
      to the fastest possible oscillations along the Measurement line,
      among all possible source positions within the Source Box.}
    \label{fig:radialsetup1_setup}
\end{subfigure}
\begin{subfigure} {0.85\textwidth}
    \centering 
    \includegraphics[width=1\textwidth]{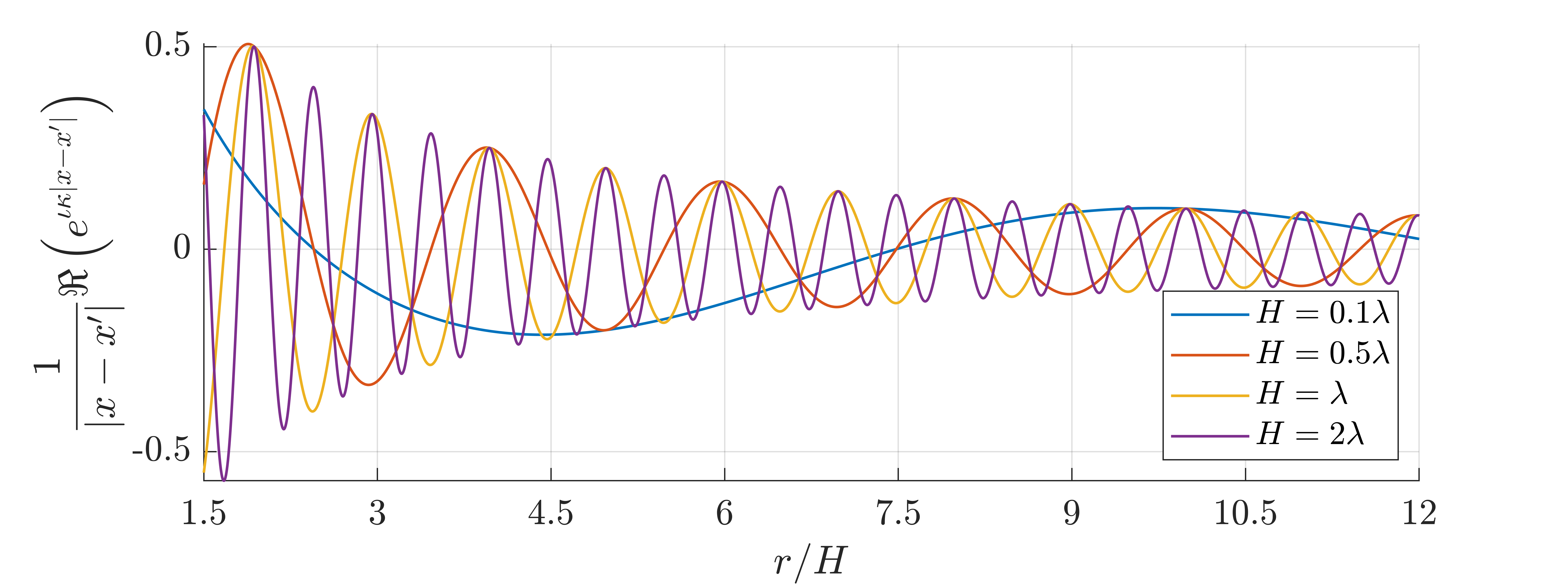}
    \caption{Real part of the Green function $G$ in
      equation~\eqref{eq:greensfunction} (without factorization),
      along the Measurement line depicted in
      Figure~\ref{fig:radialsetup1_setup}, for boxes of various
      acoustic sizes $H$.}
    \label{fig:radialsetup1_unfactored}
\end{subfigure}
\begin{subfigure} {0.85\textwidth}
    \centering 
    \includegraphics[width=\textwidth]{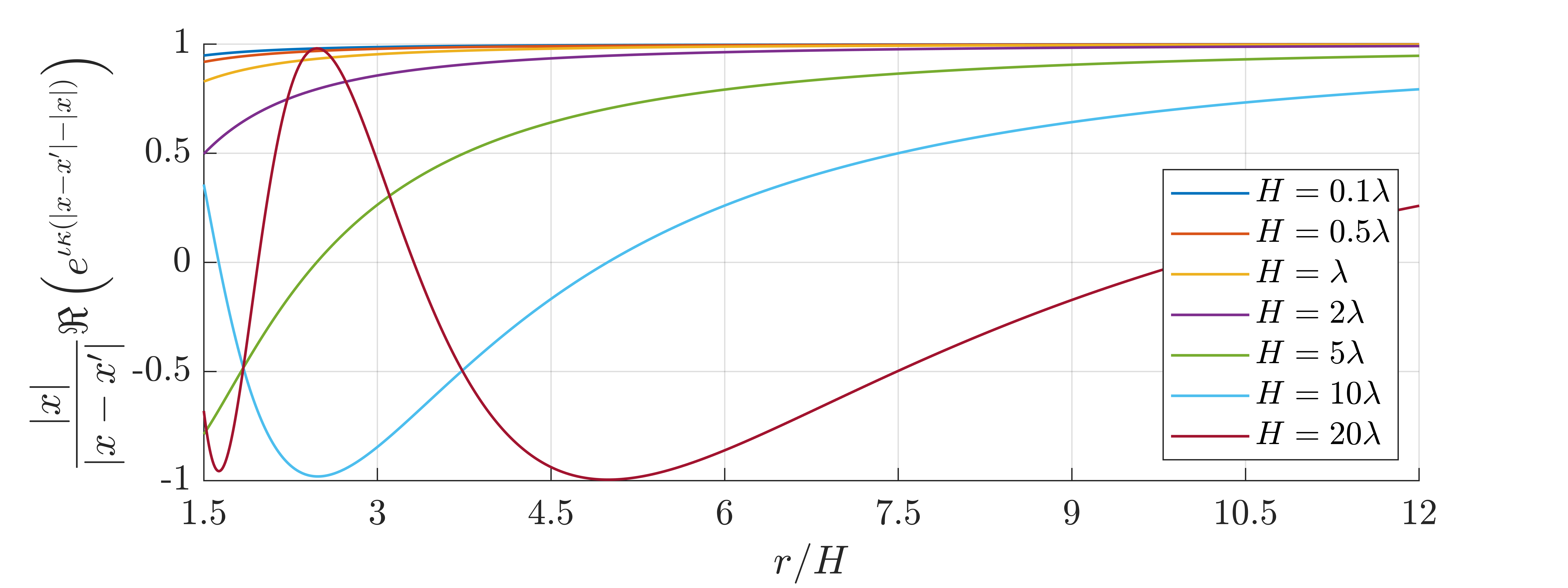}
    \caption{Real part of the analytic factor $g_S$
      (equation~\eqref{eq:factored_f}) along the Measurement line
      depicted in Figure~\ref{fig:radialsetup1_setup}, for boxes of
      various acoustic sizes $H$. }
    \label{fig:radialsetup1_factored}
\end{subfigure}
\caption{Surrogate Source factorization test, set up as illustrated in
  Figure~\ref{fig:radialsetup1_setup}. Figure~\ref{fig:radialsetup1_factored}
  shows that the analytic factor $g_S$ oscillates much more slowly,
  even for $H=20\lambda$, than the unfactored Green function does for
  the much smaller values of $H$ considered in
  Figure~\ref{fig:radialsetup1_unfactored}. }
\end{figure}
While the oscillations of the smooth factor $g_S$ and the unfactored
Green function are asymptotically the same for an increasing acoustic
size of the source box ($\kappa H \to \infty$, cf. Theorem
\ref{theorem:Error}), a strategy based on direct interpolation of the
Green function without factorization of the complex exponential term
would require several orders of magnitudes more interpolation points
and proportional computational effort. While allowing that the cost of
such an approach may be prohibitive, it is interesting to note that,
asymptotically, the cost would still be of the order of
$\mathcal{O}(N\log N)$ operations.

In addition to the factorization~\eqref{eq:definitionF}, the proposed
strategy relies on use of the singularity resolving
change of variables 
\begin{equation}\label{eq:defparametrizationins} 
{\color{red} s \coloneqq \frac{h}{r}, \qquad 
{\mathbf x}(s, \theta, \varphi)  \coloneqq  \tilde {\mathbf x}(r, \theta, \varphi),}
\end{equation}
where, once again, $r = |x|$ denotes the radius in spherical
coordinates and where $h$ denotes the radius of the source box---which
is related to the box size $H$ by
\begin{equation} \label{eq:def_eps}
    h \coloneqq \max \limits_{x \in B(x_S, H)} |x| = \frac{\sqrt{3}}{2}H.
\end{equation}
Using these notations equation~\eqref{factor_r} may be re-expressed in
the form
\begin{equation}\label{factor_r_eps}
  g_S(x, x') = \frac{1}{4 \pi \left |\frac{x}{r} - \frac{x'}{h}s\right|}  \exp\left({\imath \kappa r\left( \left |\frac{x}{r}-\frac{x'}{h} s\right| - 1\right)}\right).
\end{equation}
Note that while the source point $x$ and its norm $r$ depend on $s$, the quantity $x/r$ is independent of $r$ and therefore also of $s$.

The introduction of the variable $s$ gives rise to several algorithmic
advantages, all of which stem from the analyticity properties of the
function $g_S$---as presented in Lemma~\ref{lem:analyticity} below and
Theorem~\ref{theorem:Error} in Section~\ref{subsec:interpolation}.
Briefly, these results establish that, for any fixed values $H>0$ and
$\eta$ satisfying $0<\eta<1$, the function $g_S$ is analytic for
$(x,x')\in A_{\eta}^H$, with $x$-derivatives that are bounded up to
and including $|x|=\infty$. As a result (as shown in Section
\ref{subsec:interpolation}) the $s$ change of variables translates the
problem of interpolation of $g_S$ over an infinite $r$ interval into a
problem of interpolation of an analytic function of the variable $s$
over a compact interval in the $s$ variable.

The relevant $H$-dependent analyticity domains for the function $g_S$
for each fixed value of $H$ are described in the following lemma.
\begin{lemma} \label{lem:analyticity}
  Let $x'\in B(x_S, H)$ and let
  $x_0 = \tilde{\mathbf{x}}(r_0, \theta_0,\varphi_0) = \mathbf{x}(s_0,
  \theta_0,\varphi_0)$ ($s_0 = h/r_0$) be such that
  $(x_0,x')\in A_{\eta}^H$. Then $g_S$ is an analytic
  function of $x$ around $x_0$ and also an analytic function of
  $(s, \theta,\varphi)$ around $(s_0, \theta_0,\varphi_0)$. Further, the
  function $g_S$ is an analytic function of
  $(s, \theta,\varphi)$ (resp. $(r, \theta,\varphi)$) for
  $0\leq \theta \leq \pi$, $0\leq \varphi< 2\pi$, and for $s$ in a
  neighborhood of $s_0=0$ (resp. for $r$ in a neighborhood of
  $r_0=\infty$, including $r=r_0 = \infty$).

\end{lemma}
\begin{proof}
  The claimed analyticity of the function $g_S$ around $x_0 = \mathbf{x}(s_0, \theta_0, \varphi_0)$
  (and, thus, the analyticity of $g_S$ around $(s_0, \theta_0,\varphi_0)$)
  is immediate since, under the assumed hypothesis, the quantity
\begin{equation}\label{eq:denominator}
  \left |\frac{x}{r} - \frac{x'}{h}s\right|,
\end{equation}
does not vanish in a neighborhood of $x=x_0$. Analyticity around
$s_0=0$ ($r_0=\infty$) follows similarly since the
quantity~\eqref{eq:denominator} does not vanish around $s=s_0 = 0$.
\end{proof}

\begin{corollary} \label{corol:sufficientconditions}
    Let $H > 0$ be given. Then for all $x' \in B(x_S, H)$ the function $g_S({\mathbf x}(s,\theta, \varphi),x')$ is an analytic function of $(s,\theta,\varphi)$ for  $0 \leq s < 1$, $0\leq \theta\leq \pi$ and $0\leq \varphi< 2\pi$.
\end{corollary}
\begin{proof}
  Take $\eta\in(0,1)$. Then, for $0 \leq s \leq \eta$ we have
  $({\mathbf x}(s,\theta, \varphi), x') \in A_\eta^H$. The
  analyticity for $0\leq s\leq \eta$ follows from
  Lemma~\ref{lem:analyticity}, and since $\eta\in(0,1)$ is arbitrary,
  the lemma follows.
\end{proof}
For a given $x'\in\mathbb{R}^3$,
Corollary~\ref{corol:sufficientconditions} reduces the problem of
interpolation of the function $g_S(x,x')$ in the $x$ variable to a
problem of interpolation of a re-parametrized form of the function
$g_S$ over a bounded domain---provided that $(x,x')\in A_{\eta}^H$,
or, in other words, provided that $x$ is {\color{MyGreen} separated
  from $x'$ by a factor of at least $\eta$}, for some $\eta<1$. In
the IFGF algorithm presented in Section~\ref{subsec:algorithm},
side-$H$ boxes $B(x_S, H)$ containing sources $x'$ are considered,
with target points $x$ at a distance no less than $H$ away from
$B(x_S, H)$. Clearly, a point $(x, x')$ in such a configuration
necessarily belongs to $A_{\eta}^H$ with $\eta =
\sqrt{3}/3$. Importantly, as demonstrated in the following section,
the interpolation quality of the algorithm does not degrade as source
boxes of increasingly large side $H$ are used, as is done in the
proposed multi-level IFGF algorithm (with a single box size at each
level), leading to a computing cost per level which is independent of
the level box size $H$.

\subsection{Interpolation} \label{subsec:interpolation} On the basis
of the discussion presented in Section~\ref{subsec:analyticity}, the
present section concerns the problem of interpolation of the function
$g_S$ in the variables $(s, \theta, \varphi)$. For efficiency,
piece-wise Chebyshev interpolation in each one of these variables is
used, over interpolation intervals of respective lengths $\Delta_s$,
$\Delta_\theta$ and $\Delta_\varphi$, where, for a certain positive
integer $n_C$, {\color{MyGreen}angular coordinate intervals} of size
\[
\Delta_\theta = \Delta_\varphi =  \frac{\pi}{n_C},
\]
are utilized. Defining
\begin{equation*}
  \theta_k = k \Delta_\theta, \quad (k = 0, \ldots, n_C-1) \quad
  \mbox{and}\quad	\varphi_\ell = \ell \Delta_\varphi, \quad (\ell = 0, \ldots, 2 n_{C} -1),
\end{equation*}
as well as
\begin{equation} \label{eq:defconedomain}
     E^\varphi_{j} = [\varphi_{j-1}, \varphi_j) \quad \text{ and } \quad E^\theta_{i, j} = \begin{cases} [\theta_{n_C-1}, \pi] \quad &\text{for} \quad i = n_C, \mkern5mu j = 2 n_C \\
     (0, \Delta_\theta) \quad &\text{for} \quad i = 1, \mkern5mu j > 1
     \\ [\theta_{i-1}, \theta_i) \quad &\text{otherwise,} \end{cases}
\end{equation}
we thus obtain the mutually disjoint {\color{MyGreen}{\em interpolation cones}}
\begin{equation}\label{eq:defcone}
  \tilde{C}_{i, j} \coloneqq  \left\{ x = \tilde{\mathbf x}(r, \theta, \varphi) \, : \, r \in (0, \infty), \mkern5mu \theta \in E^\theta_{i, j}, \mkern5mu \varphi \in E^\varphi_{j} \right\},\quad (i = 1,\dots, n_C, j = 1, \ldots, 2 n_C),
\end{equation}
centered at $x_S = (0,0,0)^T$. Note that the definition
\eqref{eq:defcone} ensures that \begin{equation*} \bigcup
  \limits_{\substack{1 = 1, \ldots, n_C \\ j = 1, \ldots, 2 n_C}}
  \tilde C_{i, j} = \mathbb{R}^3 \setminus \{0\} \qquad \text{and}
  \qquad \tilde C_{i, j} \cap \tilde C_{k, l} = \emptyset \quad
  \text{for} \quad (i, j) \neq (k, l).
\end{equation*}

The proposed interpolation strategy additionally relies on a number
$n_s \in \mathbb{N}$ of disjoint radial interpolation intervals
$E_k^s$, $k = 1, \ldots, n_s$, of size $\Delta_s = \eta/n_s$, within
the IFGF $s$-variable radial interpolation domain $[0, \eta]$
(with $\eta = \sqrt{3}/3$, see Section~\ref{subsec:analyticity}). Thus, in
all, the approach utilizes an overall number
$N_C \coloneqq n_s \times n_C \times 2 n_C$ of interpolation domains
\begin{equation} \label{eq:defconedomainproduct}
    E_\gamma \coloneqq E_{\gamma_1}^s \times E_{\gamma_2}^\theta \times E_{\gamma_3}^\varphi,
\end{equation}
which we call {\em cone domains}, with
$\mathbf{\gamma} = (\gamma_1, \gamma_2, \gamma_3) \in \{1, \ldots, n_s
\} \times \{1, \ldots, n_C \} \times \{1, \ldots, 2 n_C \}$. Under the
parametrization $\mathbf{x}$ in
equation~\eqref{eq:defparametrizationins}, the cone domains yield the
{\em cone segment} sets
\begin{equation} \label{eq:defconesegments}
    C_\mathbf{\gamma} \coloneqq \{ x = \mathbf{x}(s, \theta, \varphi) \, : \, (s, \theta, \varphi) \in E_\gamma \}.
\end{equation}
{\color{red} Note that, by definition, the cone segments are mutually disjoint.}
A two-dimensional illustration of the cone domains and associated cone
segments is provided in Figure~\ref{fig:conedomaincone}.
\begin{figure}
    \centering
    \includegraphics[width=0.6\textwidth]{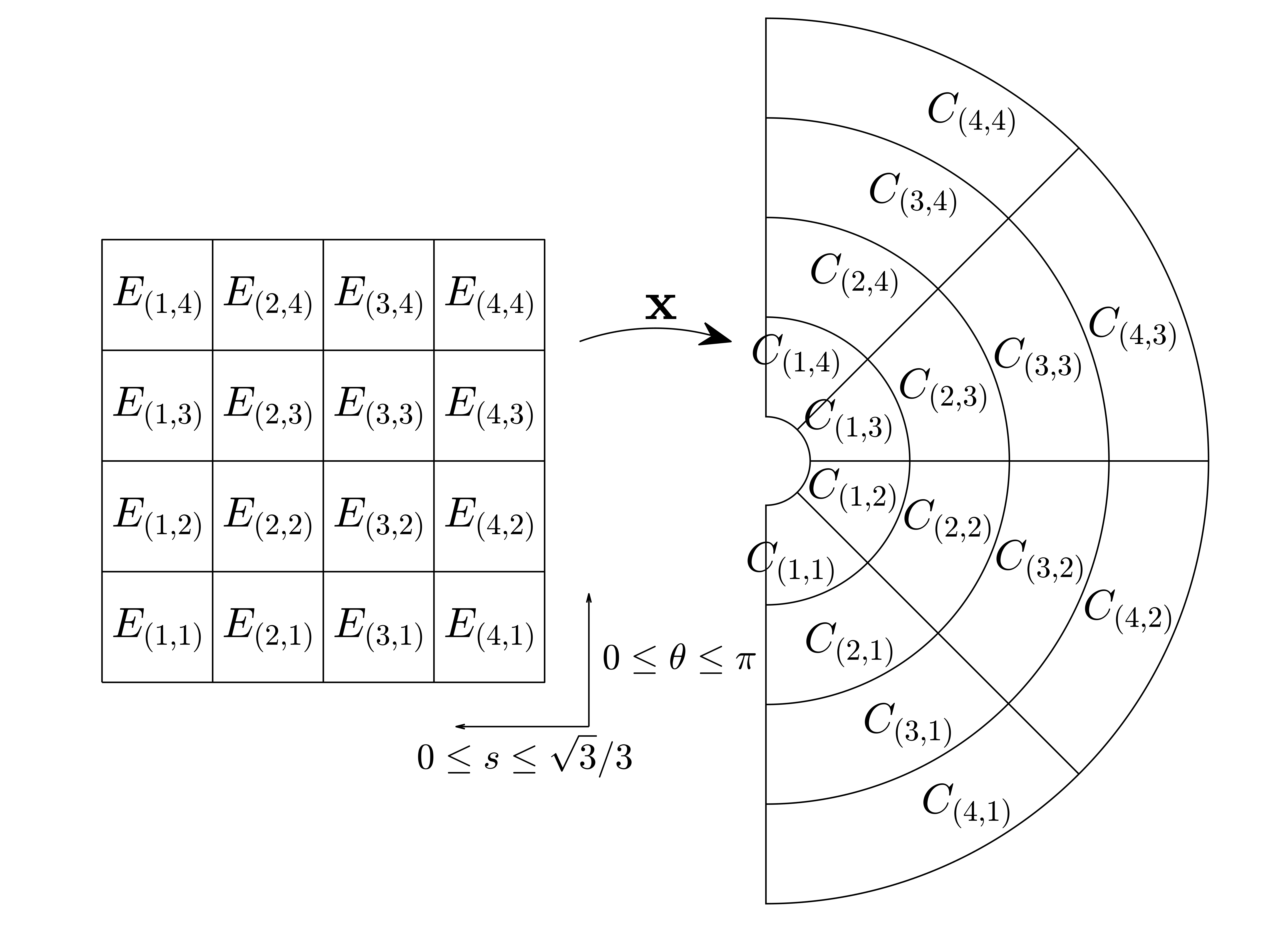}
    \caption{Schematic two-dimensional illustration of a set of cone
      domains $E_\gamma$, together with the associated cone segments
      $C_\gamma$ that result under the parametrization
      \eqref{eq:defparametrizationins}. For the sake of simplicity,
      the illustration shows constant cone-segment radial sizes (in
      the $r$ variable), but the actual radial sizes are constant in
      the $s$ variable (equation~\eqref{eq:defparametrizationins}),
      instead. Thus, increasingly large real-space cone segments are
      used as the distance of the interpolation cone segments to the
      origin grows. }
    \label{fig:conedomaincone}
\end{figure}

The desired interpolation strategy then relies on the use of a fixed
number $P = P_{\text{ang}}^2 P_{\text{s}}$ of interpolation points for
each cone segment $C_\mathbf{\gamma}$, where $P_{\text{ang}}$
(resp. $P_{\text{s}}$) denotes the number of Chebyshev interpolation
points per interval used for each angular variable (resp. for the
radial variable $s$). For each cone segment, the proposed
interpolation approach proceeds by breaking up the problem into a sequence of one-dimensional Chebyshev interpolation problems of accuracy orders $P_s$ and $P_{\text{ang}}$, as described in
\cite[Sec. 3.6.1]{NumericalRecipes2007}, along each one of the three coordinate directions $s$, $\theta$ and $\varphi$. This spherical Chebyshev interpolation procedure is described in what follows, and an associated error estimate is presented which is then used to guide the
selection of cone segment sizes.
	
The one-dimensional Chebyshev interpolation polynomial $ I^\mathrm{ref}_n u$ of
accuracy order $n$ for a given function $u : [-1, 1] \to \mathbb{C}$
over the reference interval $[-1, 1]$ is given by the
expression
\begin{equation} \label{eq:defchebyshevinterpooperator}
  I^\mathrm{ref}_n u(x) = \sum \limits_{i = 0}^{n-1} a_i T_i(x),\quad
  x \in [-1, 1],
\end{equation}
where $T_i(x) = \cos (i \arccos (x))$ denotes the $i$-th Chebyshev
polynomial of the first kind, and where, letting
\[
  x_k = \cos\left(\frac{2 k + 1}{2 n} \pi \right),\quad
    b_i = \left\{ \begin{array}{ll}
        1 \quad &i \neq 0 \\
        2 \quad &i = 0,
    \end{array}\right. \quad\mbox{and} \quad c_k = \left \{ \begin{array}{ll}
        0.5 \quad &k = 0 \text{ or } k = n-1 \\
        1 \quad &\text{else},
     \end{array}\right.
\]
the coefficients $a_i \in \mathbb{C}$ are given by
\begin{equation} \label{eq:defchebyshevcoeffsandpoints}
 a_i = \frac{2}{b_i (n-1)} \sum \limits_{k = 0}^{n-1} c_k u(x_k) T_i(x_k).
\end{equation}
Chebyshev expansions for functions defined on arbitrary intervals
$[a, b]$ result from use of a linear interval mapping to the reference
interval $[-1,1]$; for notational simplicity, the corresponding
Chebyshev interpolant in the interval $[a, b]$ is denoted by $I_n u$,
without explicit reference to the interpolation interval $[a,b]$.

As is known (\hspace{1sp}\cite[Sec. 7.1]{deuflhard2018numerische},
\cite{fox1968chebyshev}), the one-dimensional Chebyshev interpolation
error $|u(x) - I_n u(x)|$ in the interval $[a,b]$ satisfies the bound
\begin{equation} \label{eq:errorestimatechebinterpolation}
    |u(x) - I_n u(x)| \leq \frac{(b-a)^{n}}{2^{2n - 1} n! } \norm{\frac{\partial^{n} u}{\partial x^{n}}}_{\infty},
\end{equation}
where 
\begin{equation}
    \norm{\frac{\partial^{n} u}{\partial x^{n}}}_{\infty} \coloneqq \sup \limits_{c \in (a, b)} \left|\frac{\partial^{n} u}{\partial x^{n}}(c)\right|
\end{equation}
denotes the supremum norm of the $n$-th partial derivative. The desired error
estimate for the nested Chebyshev interpolation procedure within a
cone segment~\eqref{eq:defconesegments} (or, more precisely, within
the cone domains~\eqref{eq:defconedomainproduct}) is provided by the
following theorem.
\begin{theorem} \label{theorem:errorestimatenested} Let $I_{P_s}^s$,
  $I_{P_{\text{ang}}}^\theta$, and $I_{P_{\text{ang}}}^\varphi$ denote the Chebyshev interpolation operators of accuracy orders $P_s$ in
  the variable $s$ and $P_{\text{ang}}$ in the angular variables
  $\theta$ and $\varphi$, over intervals $E^s$, $E^\theta$, and
  $E^\varphi$ of lengths $\Delta_s$, $\Delta_\theta$, and
  $\Delta_\varphi$ in the variables $s$, $\theta$, and $\varphi$,
  respectively. Then, for each arbitrary but fixed point
  $x' \in \mathbb{R}^3$ the error arising from nested interpolation of
  the function $g_S(\mathbf{x}(s, \theta, \varphi), x')$
  (cf. equation~\eqref{eq:defparametrizationins}) in the variables
  $(s, \theta, \varphi)$ satisfies the estimate
  \begin{multline} \label{eq:errorestimate}
      |g_S(\mathbf{x}(s, \theta, \varphi), x') - I_{P_{\text{ang}}}^\varphi I_{P_{\text{ang}}}^\theta I_{P_s}^s g_S(\mathbf{x}(s, \theta, \varphi), x')| \leq \\
      C\left[ (\Delta_s)^{P_{\text{s}}} \norm{\frac{\partial^{P_s} g_S}{\partial s^{P_s}}}_{\infty} + (\Delta_\theta)^{P_{\text{ang}}} \norm{\frac{\partial^{P_{\text{ang}}} g_S}{\partial \theta^{P_{\text{ang}}}}}_{\infty} + (\Delta_\varphi)^{P_{\text{ang}}} \norm{\frac{\partial^{P_{\text{ang}}} g_S}{\partial \varphi^{P_{\text{ang}}}}}_{\infty} \right],
  \end{multline}
  for some constant $C$ depending only on $P_s$ and $P_{\text{ang}}$,
  where the supremum-norm expressions are shorthands for the supremum
  norm defined by
  \begin{equation*}
    \norm{\frac{\partial^{n} g_S}{\partial \xi^{n}}}_{\infty} \coloneqq \sup \limits_{\substack{\tilde s \in E^s \\ \tilde \theta \in E^\theta \\ \tilde \varphi \in E^\varphi}} \left|\frac{\partial^{n} g_S}{\partial \xi^{n}}({\mathbf x}(\tilde s, \tilde \theta, \tilde \varphi), x')\right|
\end{equation*}
for $\xi = s$, $\theta$, or $\varphi$.
\end{theorem}
\begin{proof}
  The proof is only presented for a double-nested interpolation
  procedure; the extension to the triple-nested method is entirely
  analogous. Suppressing, for readability, the explicit functional
  dependence on the variables $x$ and $x'$, use of the triangle
  inequality and the error 
  estimate~\eqref{eq:errorestimatechebinterpolation} yields
  \begin{align*}
     |g_S -  I_{P_{\text{ang}}}^\theta I_{P_s}^s g_S| &\leq |f -  I_{P_s}^s g_S| + |I_{P_{\text{ang}}}^\theta I_{P_s}^s g_S - I_{P_s}^s g_S|  \\
     &\leq C_1 (\Delta_s)^{P_s} \norm{\frac{\partial^{P_s} g_S}{\partial s^{P_s}}}_\infty + C_2 (\Delta_\theta)^{P_{\text{ang}}} \norm{\frac{\partial^{P_{\text{ang}}} I_{P_s}^s g_S}{\partial \theta^{P_{\text{ang}}}}}_\infty,
  \end{align*}
  where $C_1$ and $C_2$ are constants depending on $P_s$ and
  $P_{\text{ang}}$, respectively. In order to estimate the second term
  on the right-hand side in terms of derivatives of $g_S$ we utilize
  equation~\eqref{eq:defchebyshevcoeffsandpoints} in the shifted
  arguments corresponding to the $s$-interpolation interval $(a,b)$:
\[
    I_{P_s}^s g_S = \sum \limits_{i = 0}^{P_s-1} a_i^s(\theta)
    T_i\left(2\frac{s - a}{b - a} - 1\right),\quad (b=a+\Delta_s).
\]
Differentiation with respect to $\theta$ and use of the relations
\eqref{eq:defchebyshevinterpooperator} and
\eqref{eq:defchebyshevcoeffsandpoints} then yield
\begin{equation*}
      \norm{\frac{\partial^{P_{\text{ang}}} I_{P_s}^s g_S}{\partial \theta^{P_{\text{ang}}}}}_\infty 
      \leq P_s \max \limits_{i = 1, \ldots, P_s -1} \norm{\frac{\partial^{P_{\text{ang}}} a_i^s}{\partial \theta^{P_{\text{ang}}}}}_\infty
      \leq C_3 \norm{\frac{\partial^{P_{\text{ang}}} g_S}{\partial \theta^{P_{\text{ang}}}}}_\infty,
    \end{equation*}
    as it may be checked, for a certain constant $C_3$ depending on
    $P_s$, by employing the triangle inequality and the $L^\infty$
    bound $\norm{T_i}_\infty \leq 1$ ($i\in \mathrm{\mathbb{N}_0 = \mathbb{N}\cup\{0\}}$). The
    more general error estimate~\eqref{eq:errorestimate} follows by a
    direct extension of this argument to the triple-nested case, and
    the proof is thus complete.
\end{proof}

The analysis presented in what follows, including
Lemmas~\ref{lem:derivatives} through~\ref{lem:dexponential_s_estimate}
and Theorem~\ref{theorem:Error}, yields bounds for the partial
derivatives in \eqref{eq:errorestimate} in terms of the acoustic size
$\kappa H$ of the source box $B(x_S, H)$. Subsequently, these bounds
are used, together with the error estimate~\eqref{eq:errorestimate},
to determine suitable choices of the cone domain sizes $\Delta_s$,
$\Delta_\theta$, and $\Delta_\varphi$, ensuring that the errors
resulting from the triple-nested interpolation process lie below a
prescribed error tolerance. Leading to Theorem~\ref{theorem:Error},
the next three lemmas provide estimates, in terms of the box size $H$,
of the $n$-th order derivatives ($n\in \mathbb{N}$) of certain
functions related to $g_S(\mathbf{x}(s, \theta, \varphi) ,x')$, with
respect to each one of the variables $s$, $\theta$, and $\varphi$ and
every $x' \in B(x_S, H)$.
{\color{MyGreen} \begin{lemma} \label{lem:derivatives} Under the change of variables
  $x = {\mathbf x}(s, \theta, \varphi)$
  in~\eqref{eq:defparametrizationins}, for all $n \in \mathbb{N}$ and
  for either $\xi=\theta$ or $\xi=\varphi$, we have
    \begin{equation*}
        \frac{\partial^n}{\partial \xi^n}|x - x'| = \sum \frac{c(m_1, \ldots, m_n)}{|x-x'|^{2 k - 1}} \prod \limits_{j = 1}^n \left \langle \frac{\partial^j x}{\partial \xi^j}, x' \right \rangle^{m_j},
    \end{equation*}
    where the outer sum is taken over all $n$-tuples
$(m_1, \ldots, m_n)\in \mathbb{N}_0^{n}$ such that 
    \begin{equation*} 
        \sum \limits_{j = 1}^n j m_{j} = n,
    \end{equation*}
    where $k := \sum_{i=1}^n m_i$,
    where $c(m_1, \ldots, m_n) \in \mathbb{R}$ denote constants independent of $x$,
    $x'$ and $\xi$, and where $\langle \cdot, \cdot \rangle$ denotes the Euclidean
    inner product on $\mathbb{R}^3$.
\end{lemma}
\begin{proof}
  The proof follows from Fa\`a di Bruno's formula \cite{Bruno1857}
  applied to $f(g(x)) = |x-x'|$, where $f(x) = \sqrt{x}$ and
  $g(x) = \langle x, x \rangle - 2 \langle x, x' \rangle + \langle x',
  x' \rangle$. Indeed, noting that
  \begin{equation*}
      \frac{d^k f(x)}{d x^k} = c_1(k) \frac{1}{f(x)^{2 k -1}}, 
  \end{equation*}
  for some constant $c_1(k)$, and that, since
  $\langle \frac{\partial x}{\partial \xi}, x \rangle = 0$ for
  $\xi=\theta$ and $\xi=\varphi$,
  \begin{equation*}
      \frac{\partial^i g(x(\xi))}{d \xi^i} = c_2(i) \left \langle \frac{\partial^i x}{\partial \xi^i}, x' \right \rangle,
  \end{equation*}
  for some constant $c_2(i)$, an application of Fa\`a di Bruno's
  formula directly yields the desired result.
\end{proof} }
\begin{lemma} \label{lem:exponentsestimate} Let $H>0$ and
  $\eta\in(0,1)$ be given. Then, under the change of variables
  $x = {\mathbf x}(s, \theta, \varphi)$
  in~\eqref{eq:defparametrizationins},  the exponent in the
  right-hand exponential in~\eqref{eq:factored_f} satisfies
    \begin{equation*}
    \frac{\partial^n}{\partial \xi^n} \left( |x - x'| - |x|\right) \leq C(\eta, n) H,
  \end{equation*}
  for all $(x, x') \in A_{\eta}^H$, for all
  $n \in \mathbb{N}_0$, and for
$\xi=s$, $\xi=\theta$ and $\xi=\varphi$, where $C(\eta, n)$ is a
certain real constant that depends on $\eta$ and $n$, but which is
independent of $H$.
\end{lemma}
\begin{proof}
  Expressing the exponent in~\eqref{eq:factored_f} in terms of $s$
  yields
\begin{align} \label{eq:exponentins}
|x - x'| - |x| = \frac{h}{s}\left( \left \vert \frac{x}{r} - \frac{x'}{h} s \right \vert - 1 \right) =: h g(s),
\end{align}
where our standing assumption $x_S = 0$ and notation $|x| =r$ have
been used (so that, in particular, $x/r$ is independent of $r$ and
therefore also independent of $s$), and where the angular dependence
of the function $g$ has been suppressed. Clearly,
$g(s)$ is an analytic function of $s$ for $s\in \big[0, h/|x'|\big)$
and, thus, since $\eta< 1$, for $s$ in the compact interval
$\big[0, \eta\cdot h/|x'|\big]$. It follows that $g$ and each one of
its derivatives with respect to $s$ is uniformly bounded for all
$s\in \big[0, \eta\cdot h/|x'|\big]$ and (as shown by a simple
re-examination of the discussion above) for all $H$ and for all
values of $x/r$ and $x'/h$ under consideration. Since at the point
$(x,x')$ we have
$s= h/|x| = |x'|/|x||\cdot h/|x'\leq \eta\cdot h/|x'|$,
using~\eqref{eq:def_eps} once again, the desired $\xi = s$ estimate
\begin{equation*}
    \frac{\partial^n}{\partial s^n} \left( h g(s) \right) \leq C(\eta, n) H,
\end{equation*}
follows, for some constant $C(\eta, n)$.

Turning to the angular variables, we only consider the case
$\xi = \theta$; the case $\xi =\varphi$ can be treated similarly. Using
Lemma~\ref{lem:derivatives} for $\xi = \theta$, the Cauchy Schwarz
inequality and the assumption $(x, x') \in A_\eta^H$, we
obtain
{\color{MyGreen}
\begin{align*}
    \left| \frac{\partial^n \left( |x - x'| - |x| \right)}{\partial \theta^n} \right| &= \left| \frac{\partial^n \left( |x - x'| \right)}{\partial \theta^n} \right| = \left \vert\sum \frac{c(m_1, \ldots, m_n)}{|x-x'|^{2 k - 1}} \prod \limits_{j = 1}^n \left \langle \frac{\partial^j x}{\partial \xi^j}, x' \right \rangle^{m_j} \right\vert \\
    &\leq \sum \frac{|c(m_1, \ldots, m_n)|}{|x - x'|^{2 k - 1}} \prod \limits_{j = 1}^n \left | \frac{\partial^j x}{\partial \xi^j} \right |^{m_j} \left|x'\right|^{m_j} \leq \sum \limits_{k = 1}^n \hat{C}(\eta, n)\frac{1}{r^{2 k - 1}} r^k \left|x'\right|^k \\ &\leq \tilde{C}(\eta, n) \left|x'\right| \leq C(\eta, n) H,
\end{align*}
where the same notation as in Lemma~\ref{lem:derivatives} was used.
The constant $C(\eta, n)$ has been suitably adjusted. The proof is now
complete.}
\end{proof}
																																 
\begin{lemma} \label{lem:dexponential_s_estimate} Let $H>0$ and
  $\eta\in(0,1)$ be given. Then, under the change of variables
  $x = {\mathbf x}(s, \theta, \varphi)$
  in~\eqref{eq:defparametrizationins}, for all
  $(x, x') \in A_{\eta}^H$, for all $n \in \mathbb{N}_0$,
  and for $\xi=s$, $\xi=\theta$ and $\xi=\varphi$, we have
\begin{equation*}
\left| \frac{\partial^n}{\partial \xi^n} e^{\imath \kappa \left( |x - x'| - |x|\right)} \right| \leq \tilde M(\eta, n) \left(\kappa H\right)^n,
\end{equation*}
where $\tilde M(\eta, n)$ is a certain real constant that depends on
$\eta$ and $n$ but which is independent of  $H$.
\end{lemma}
\begin{proof}
    Using Fa\`a di Bruno's formula \cite{Bruno1857} yields
\begin{equation*}
    \frac{\partial^n}{\partial \xi^n} e^{\imath \kappa \left( |x - x'| - |x|\right)} = \sum c(m_1, \ldots, m_n) e^{\imath \kappa \left( |x - x'| - |x|\right)} \prod \limits_{j = 1}^n \left( \imath \kappa \frac{\partial^j \left( |x - x'| - |x|\right)}{\partial \xi^j} \right)^{m_j},
\end{equation*}
where the sum is taken over all $n$-tuples
$(m_1, \ldots, m_n)\in \mathbb{N}_0^{n}$
such that
\begin{equation*}
  \sum \limits_{j = 1}^n j m_j = n,
\end{equation*}
and where $c(m_1, \ldots, m_n)$ are certain constants which depend on
$m_1, \ldots, m_n$. Using the triangle inequality and Lemma
\ref{lem:exponentsestimate} then completes the proof.
\end{proof}

The desired bounds on derivatives of the function $g_S$ are presented
in the following theorem.
\begin{theorem} \label{theorem:Error} Let $H>0$ and $\eta\in(0,1)$ be
  given. Then, under the change of variables
  $x = {\mathbf x}(s, \theta, \varphi)$
  in~\eqref{eq:defparametrizationins}, for all
  $(x, x') \in A_{\eta}^H$, for all $n \in \mathbb{N}_0$, and
  for $\xi=s$, $\xi=\theta$ and $\xi=\varphi$, we have
\begin{equation*}
    \left| \frac{\partial^n g_S}{\partial \xi^n} \right| \leq M(\eta, n) \max{\left\{(\kappa H)^n, 1\right\}},
\end{equation*}
where $M(\eta, n)$ is a certain real constant that depends on $\eta$
and $n$ but which is independent of $H$.
\end{theorem}
\begin{proof}
  The quotient on the right-hand side of~\eqref{eq:factored_f} may be
  re-expressed in the form
\begin{align} \label{eq:factorins}
	\frac{|x|}{|x -  x'|} = \frac{1}{\left \vert \frac{x}{r} - \frac{ x'}{h} s\right \vert},
\end{align}
where $x/r$ is independent of $r$ and therefore also independent of
$s$. An analyticity argument similar to the one used in the proof of
Lemma~\ref{lem:exponentsestimate} shows that this quotient, as well as
each one of its derivatives with respect to $s$, is uniformly bounded
for $s$ throughout the interval $\big[0, \eta\cdot h/|x'|\big]$, for
all $H>0$, and for all relevant values of $x/r$ and $x'/h$.

In order to obtain the desired estimates we now utilize Leibniz'
differentiation rule, which yields
\begin{align*}
    \left \vert \frac{\partial^n g_S(x,  x')}{\partial \xi^n} \right\vert = \left \vert\sum \limits_{i = 0}^n \binom{n}{i} \frac{\partial^{n-i}}{\partial \xi^{n-i}} \left(\frac{|x|}{|x- x'|} \right)\frac{\partial^i}{\partial \xi^i} \left( e^{\imath \kappa \left( |x- x'| - |x|\right)} \right) \right\vert \leq C(\eta,n) \sum \limits_{i = 0}^n \frac{\partial^i}{\partial \xi^i} e^{\imath \kappa \left( |x- x'| - |x|\right)},
\end{align*}
for some constant $C(\eta,n)$ that depends on $\eta$ and $n$, but
which is independent of $H$. Applying
Lemma~\ref{lem:dexponential_s_estimate} and suitably adjusting
constants the result follows.
\end{proof}

{\color{red} In view of the bound~\eqref{eq:errorestimate}},
Theorem~\ref{theorem:Error} shows that the interpolation error remains
uniformly small provided that the interpolation interval sizes
$\Delta_s$, $\Delta_\theta$, and $\Delta_\varphi$ {\color{MyGreen} are
  held constant for $\kappa H < 1$ and are taken to decrease like
  $\mathcal{O}(1/(\kappa H) )$ as the box sizes $\kappa H$ grow when
  $\kappa H \geq 1$}.

This observation motivates the main strategy in the IFGF algorithm. As
the algorithm progresses from one level to the next, the box sizes are
doubled, from $H$ to $2 H$, and the cone segment interpolation
interval lengths $\Delta_s$, $\Delta_\theta$, and $\Delta_\varphi$ are
{\color{MyGreen} either kept constant or decreased by a factor of
  $1/2$ (depending on whether $\kappa H<1$ or $\kappa H\geq 1$,
  respectively)}---while the interpolation error, at a fixed number of
degrees of freedom per cone segment, remains uniformly bounded. The
resulting hierarchy of boxes and cone segments is embodied in two
different but inter-related hierarchical structures: the box octree
and a hierarchy of cone segments. In the box octree each box contains
eight equi-sized child boxes. In the cone segment hierarchy,
similarly, each cone segment (spanning certain angular and radial
intervals) spawns {\em up to} eight child segments. The
$\kappa H \to \infty$ limit then is approached as the box tree
structure is traversed from children to parents and the accompanying
cone segment structure is traversed from parents to children. This
hierarchical strategy and associated structures are described in
detail in Section~\ref{subsec:algorithm}.

{\color{red} The properties of the proposed interpolation strategy, as
  implied by Theorem~\ref{theorem:Error} (in presence of
  Theorem~\ref{theorem:errorestimatenested}), are illustrated by the
  blue dash-dot error curves presented on the right-hand plot in
  Figure~\ref{fig:RadFacIncreasingBoxSizeDiffMethods}. For reference,
  this figure also includes error curves corresponding to various
  related interpolation strategies, as described below. In this
  demonstration the field generated by one thousand sources randomly
  placed within a source box $B(x_S, H)$ of acoustic size $\kappa H$
  is interpolated to one thousand points randomly placed within a cone
  segment of interval lengths $\Delta_s$, $\Delta_\theta$, and
  $\Delta_\varphi$ proportional to $\min\{1,1/(\kappa H)\}$---which,
  in accordance with Theorems~\ref{theorem:errorestimatenested}
  and~\ref{theorem:Error}, ensures essentially constant errors.  All
  curves in Figure~\ref{fig:RadFacIncreasingBoxSizeDiffMethods} report
  errors relative to the maximum absolute value of the exact
  one-thousand source field value within the relevant cone segment.
  The target cone segment used is symmetrically located around the $x$
  axis, and it lies within the $r$ range
  $3 H / 2 \leq r \leq 3 H / 2 + \Delta_r$, for the value
\begin{equation*}
    \Delta_r = \frac{9 H \Delta_s}{2 \sqrt{3}(1 - \sqrt{3} \Delta_s)}
\end{equation*}
corresponding to a given value of $\Delta_s$. It is useful to note
that, depending on the values of $\theta$ and $\varphi$, the distance
from the closest possible singularity position to the left endpoint of
the interpolation interval could vary from a distance of $H$ to a
distance of $\frac{\sqrt{3}{(\sqrt{3}-1)}}{2}H\approx 0.634 H$; cf.
Figure~\ref{fig:radialsetup1_setup}. In all cases the interpolations
were produced by means of Chebyshev expansions of degree two and four
(with numerical accuracy of orders $P_s = 3$ and $P_\text{ang} = 5$)
in the radial and angular directions, respectively.  The
($\kappa H$-dependent) radial interpolation interval sizes $\Delta_s$
were selected as follows: starting with the value
$\Delta_s = \sqrt{3}/3$ for $\kappa H = 10^{-1}$, $\Delta_s$ was
varied proportionally to $1/(\kappa H)$ (resp.
$\min\{1,1/(\kappa H)\}$) in the left-hand (resp. right-hand) plot as
$\kappa H$ increases. (Note that the value $\Delta_s=\sqrt{3}/3$,
which corresponds to the infinite-length interval going from
$r = 3H/2$ to $r = \infty$, is the maximum possible value of
$\Delta_s$ along an interval on the $x$ axis whose distance to the
source box is not smaller than one box-size $H$. In particular, the
errors presented for $\kappa H = 10^{-1}$ correspond to interpolation,
using a finite number of intervals, along the entire rightward $x$
semi-axis starting at $x = 3H/2$.)  The corresponding angular
interpolation lengths $\Delta_\theta = \Delta_\varphi$ were set to
$\pi/4$ for the initial $\kappa H = 10^{-1}$ value, and they were then
varied like the radial interval proportionally to $1/(\kappa H)$
(resp.  $\min\{1,1/(\kappa H)\}$) in the left-hand (resp. right-hand)
plot.}

As indicated above, the figure shows various interpolation results,
including results for interpolation in the variable $r$ without
factorization {\color{MyGreen}(thus interpolating the Green function
  \eqref{eq:greensfunction} directly)}, with exponential factorization
{\color{MyGreen}(factoring only $\exp{(\iota \kappa |x|)}$ and
  interpolating $\exp{(\imath \kappa (|x - x'| - |x|)}/r$)}, with
exponential and denominator factorization (called full factorization,
{\color{MyGreen} factoring the centered factor interpolating the
  analytic factor as in \eqref{eq:factored_f}}), and, finally, for the
interpolation in the $s$ variable also under full factorization. It
can be seen that the exponential factorization is beneficial for the
interpolation strategy in the {\em high frequency regime} ($\kappa H $
large) while the factorization of the denominator and the use of the
$s$ change of variables is beneficial for the interpolation in the
{\em low frequency regime} ($\kappa H $ small). Importantly, the
right-hand plot in Figure~\ref{fig:RadFacIncreasingBoxSizeDiffMethods}
confirms that, as predicted by theory, constant interval sizes in all
three variables $(s, \theta, \varphi)$ suffice to ensure a constant
error in the low frequency regime. Thus, the overall strategy leads to
constant errors for $0\leq \kappa H
<\infty$. Figure~\ref{fig:RadFacIncreasingBoxSizeDiffMethods} also
emphasizes the significance of the factorization of the denominator,
i.e. the removal of the singularity, without which interpolation with
significant accuracy would be only achievable using a prohibitively
large number of interpolation points. And, it also shows that the
change of variables from the $r$ variable to the $s$ variable leads to
a selection of interpolation points leading to improved accuracies
for small values of $\kappa H$.
\begin{figure}
\centering
\begin{subfigure}{0.48\textwidth}
\includegraphics[width=\textwidth]{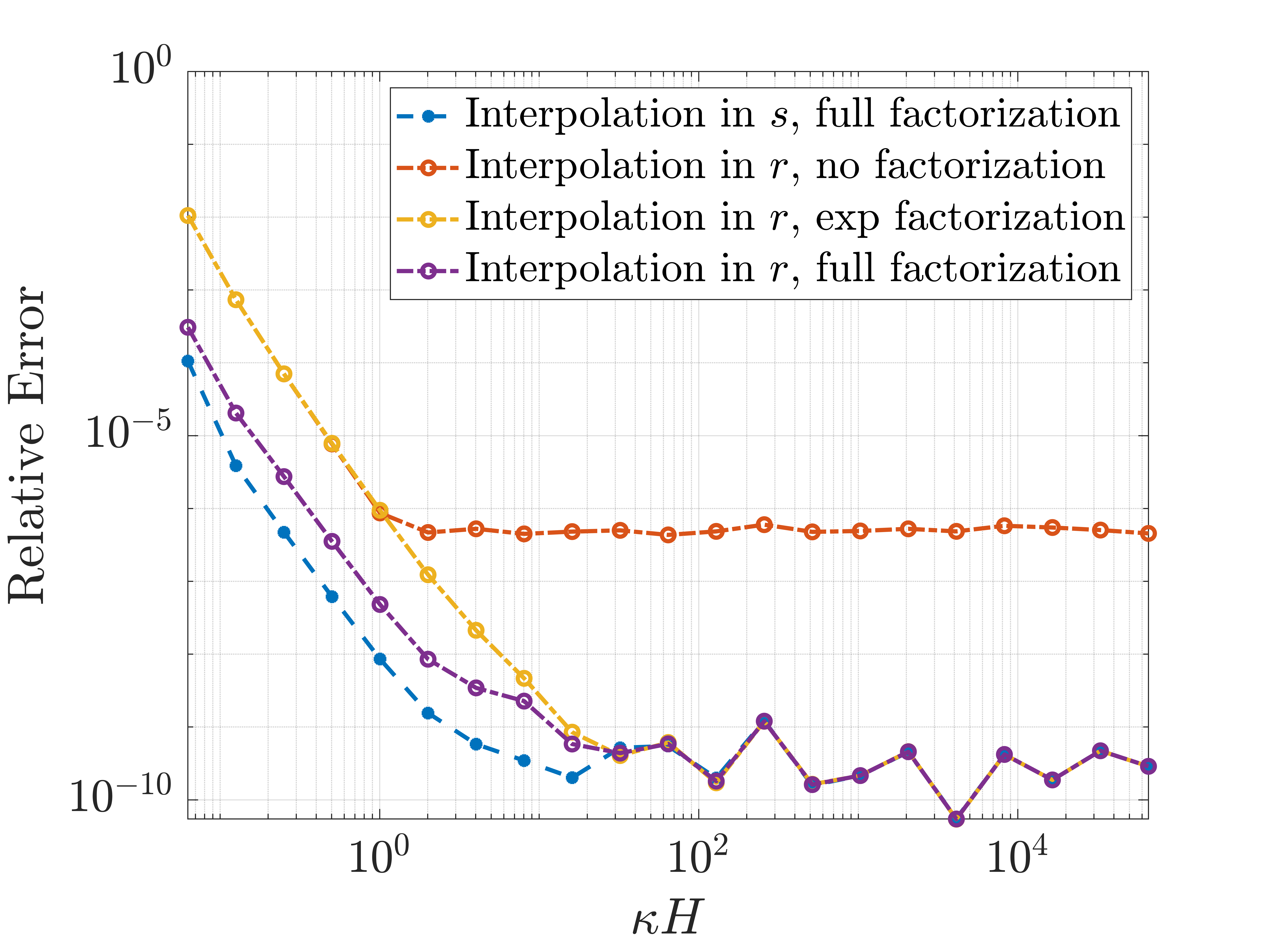}

\end{subfigure} \hspace{0.02\textwidth}
\begin{subfigure}{0.48\textwidth}
\includegraphics[width=\textwidth]{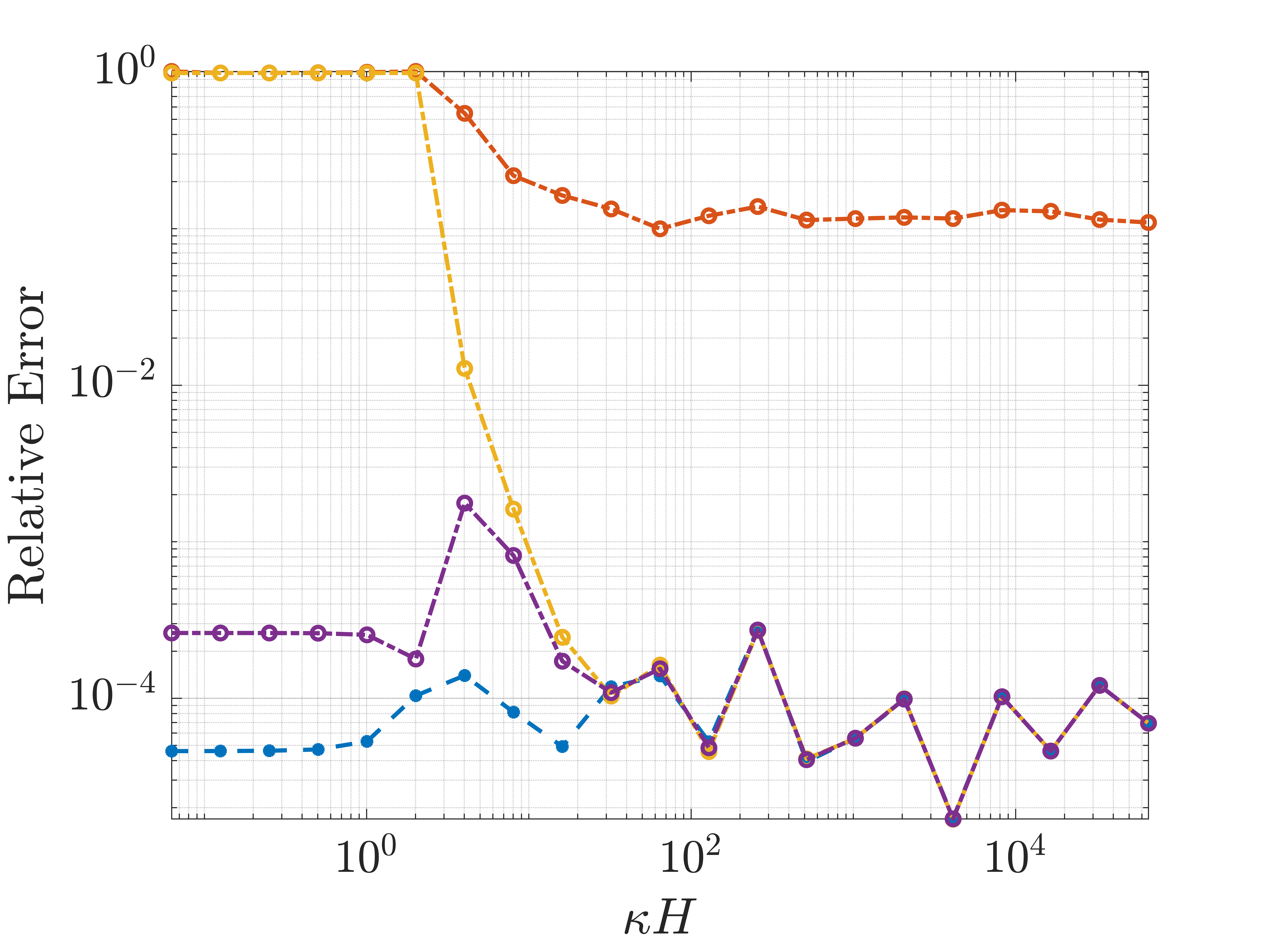}
\end{subfigure}
\caption{{\color{MyGreen} Numerical investigation of Theorem~\ref{theorem:Error} showing the overall interpolation error for various Green function
  factorizations and two different cone segment refinement strategies. Left graph: Errors resulting from use of
  interpolation intervals of sizes $\Delta_s$, $\Delta_\theta$ and
  $\Delta_\varphi$ proportional to $1/(\kappa H)$---which suffices
  to capture the oscillatory behavior for large $\kappa H$, but
  which under-resolves the singularity that arises for small
  $\kappa H$ values, for which the Green function singular point
  $x=x'$ is approached. Right graph: Errors resulting from use of
  interpolation interval sizes $\Delta_s$, $\Delta_\theta$ and
  $\Delta_\varphi$ that remain constant for small $\kappa H$ {\color{MyGreen} ($< 1$)}, and
  which decrease like $1/(\kappa H)$ for large $\kappa H$ {\color{MyGreen} ($> 1$)},
  resulting in essentially uniform accuracy for all box sizes
  provided the full IFGF factorization is used. Note that the
  combined use of full factorization and interpolation in the $s$
  variable, yields the best (essentially uniform) approximations.}}
  \label{fig:RadFacIncreasingBoxSizeDiffMethods}
\end{figure}

Theorem \ref{theorem:Error} also holds for the special $\kappa = 0$
case of the Green function for the Laplace equation. In view of its
independent importance, the result is presented, in
Corollary~\ref{corol:errorlaplace}, explicitly for the Laplace case,
without reference to the Helmholtz kernel.
\begin{corollary} \label{corol:errorlaplace} Let
  $G^\Delta(x, x') = 1/|x-x'|$ denote the Green function of the
  three dimensional Laplace equation and let
  $g_S^\Delta(x, x') = |x|/|x-x'|$ be denote the analytic kernel (cf.
  equations \eqref{eq:factor} and~\eqref{eq:factored_f} with
  $\kappa =0$). Additionally, let $H>0$ and $\eta\in(0,1)$ be
  given. Then, under the change of variables
  $x = {\mathbf x}(s, \theta, \varphi)$ in
  \eqref{eq:defparametrizationins}, for all
  $(x, x') \in A_{\eta}^H$, for all $n \in \mathbb{N}_0$, and
  for $\xi=s$, $\xi=\theta$ and $\xi=\varphi$, we have
  \begin{equation}
      \left| \frac{\partial^n g^\Delta_S}{\partial \xi^n} \right| \leq M(\eta, n),
  \end{equation}
  where $M(\eta, n)$ is a certain real constant that depends on $\eta$
  and $n$ but which is independent of $H$.
\end{corollary}
Corollary~\ref{corol:errorlaplace} shows that an even simpler and more
efficient strategy can be used for the selection of the cone segment
sizes in the Laplace case. Indeed, in view of
Theorem~\ref{theorem:errorestimatenested}, the corollary tells us that
(as illustrated in Table~\ref{table:timingsLaplace}) a constant number
of cone segments per box, independent of the box size $H$, suffices to
maintain a fixed accuracy as the box size $H$ grows (as is also the
case for the Helmholtz equation for small values of $\kappa$). As
discussed in Section~\ref{sec:examples}, this reduction in complexity
leads to significant additional efficiency for the Laplace case.

Noting that Theorem~\ref{theorem:Error} implies, in particular, that
the function $g_S$ and all its partial derivatives with respect to the
variable $s$ are bounded as $s \to 0$, below in this section we
compare the interpolation properties in the $s$ and $r$ variables, but
this time in the case in which the source box is fixed and $s \to 0$
($r \to \infty$). To do this we rely in part on an upper bound on the
derivatives of $g_S$ with respect to the variable $r$, which is
presented in Corollary~\ref{corol:derivativeinr}.
\begin{corollary} \label{corol:derivativeinr} Let $H>0$ and
  $\eta\in(0,1)$ be given. Then, under the change of variables
  $x = {\mathbf x}(s, \theta, \varphi)$
  in~\eqref{eq:defparametrizationins} and for all
  $(x, x') \in A_{\eta}^H$, for all $n \in \mathbb{N}_0$ we
  have
    \begin{equation*}
          \left| \frac{\partial^n g_S}{\partial r^n} \right| \leq C_r(n, \kappa, H) \frac{1}{r^n} \sum \limits_{m \in I} \left(\frac{h}{r}\right)^m,
    \end{equation*}
    where $I$ denotes a subset of $\{1, \ldots, n\}$ including $1$.
\end{corollary}
\begin{proof} Follows directly using Theorem~\ref{theorem:Error} and
  applying Fa\`a di Bruno's formula to the composition
  $g_S(s(r), \theta, \varphi)$.
\end{proof}
Theorem~\ref{theorem:errorestimatenested}, Theorem~\ref{theorem:Error}
and Corollary~\ref{corol:derivativeinr} show that, for any fixed
value $\kappa H$ of the acoustic source box size, the error arising
from interpolation using $n$ interpolation points in the $s$ variable
(resp. the $r$ variable) behaves like $(\Delta_s)^n$ (resp.
$(\Delta_r)^n/r^{n+1}$). Additionally, as is
easily checked, the increments $\Delta_s$ and $\Delta_r$ are related
by the identity
\begin{equation} \label{eq:relationdeltasandr} \Delta_r =
  \frac{ r_0^2\Delta_s}{h - r_0
    \Delta_s },
\end{equation}
where $h$ and $r_0$ denote the source box radius \eqref{eq:def_eps}
and the left endpoint of a given interpolation interval
$r_0 \leq r \leq r_0 + \Delta_r$, respectively. These results and
estimates lead to several simple but important conclusions. On one
hand, for a given box size $\kappa H$, a partition of the
$s$-interpolation interval $[0, \eta]$ on the basis of a finite number
of equi-sized intervals of fixed size $\Delta_s$ (on each one of which
$s$-interpolation is to be performed) provide a natural and
essentially optimal methodology for interpolation of the uniformly
analytic function $g_S$ up to the order of accuracy desired. Secondly,
such covering of the $s$ interpolation domain $[0,\eta]$ by a finite
number of intervals of size $\Delta_s$ is mapped, via
equation~\eqref{eq:defparametrizationins}, to a covering of a complete
semi-axis in the $r$ variable and, thus, one of the resulting $r$
intervals must be infinitely large---leading to large interpolation
errors in the $r$ variable. Finally, values of $\Delta_r$ leading to
constant interpolation error in the $r$ variable necessarily requires
use of infinitely many interpolation intervals and is therefore
significantly less efficient than the proposed $s$ interpolation
approach.

Figure~\ref{fig:interperror} displays interpolation errors for both
the $s$- and $r$-interpolation strategies, for increasing values of
the left endpoint $r_0$ and a constant source box one wavelength in
side. The interval $\Delta_s$ is kept constant and $\Delta_r$ is taken
per equation~\eqref{eq:relationdeltasandr}. The rightmost points in
Figure \ref{fig:interperror} are close to the singular point
$r_0 = h/ \Delta_s$ of the right-hand side in
\eqref{eq:relationdeltasandr}. The advantages of the $s$-variable
interpolation procedure are clearly demonstrated by this figure.
\begin{figure}
    \centering
    \includegraphics[width=0.5\textwidth]{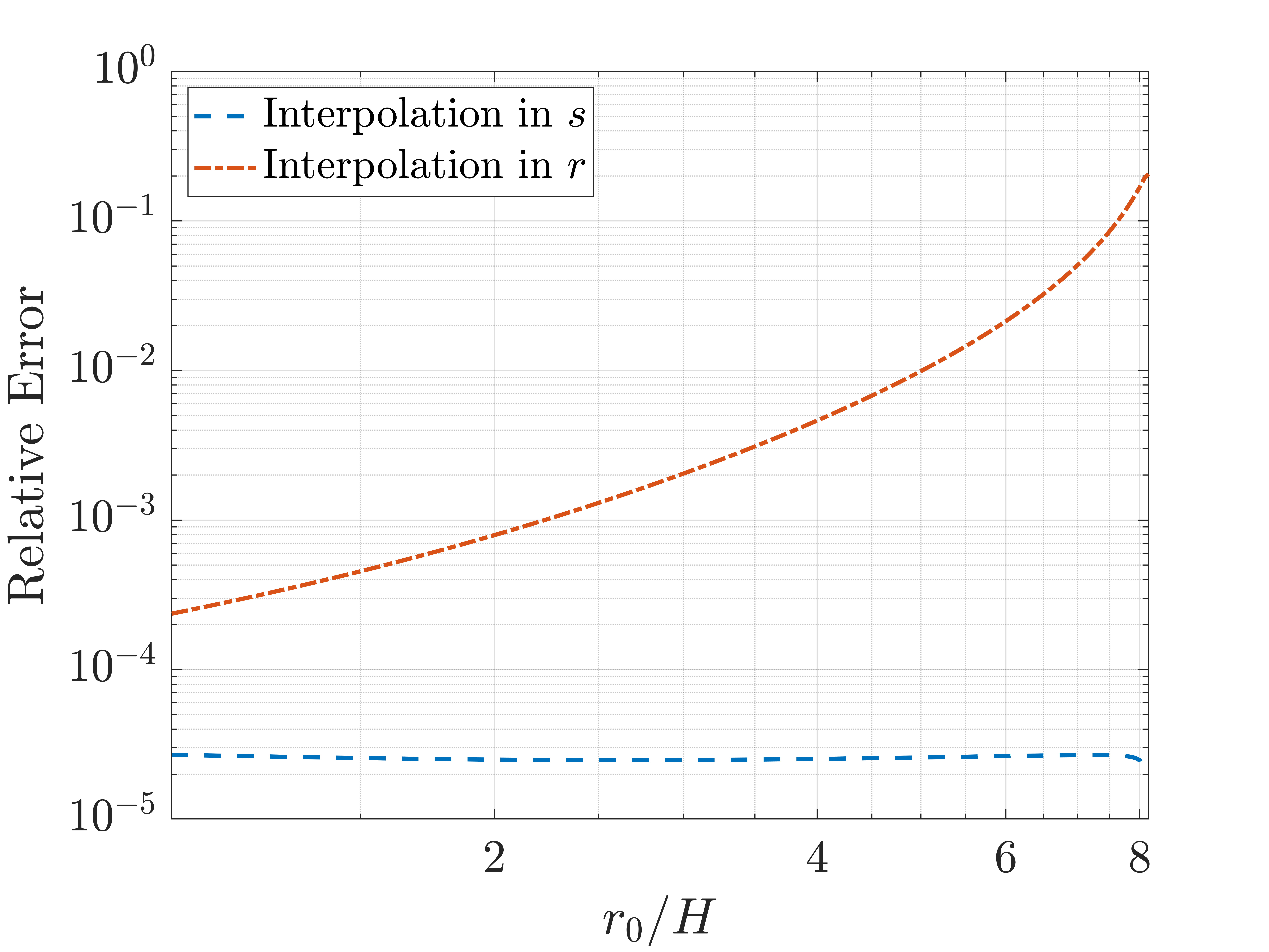}
    \caption{Comparison of the errors resulting from $r$- and
      $s$-based interpolation strategies for the problem of
      interpolation of the analytic factor $g_S$ in the interval
      $[r_0,r_0+\Delta_r)$, as a function of $r_0$. Clearly, the
      equi-spaced $s$ discretization used is optimally suited for the
      interpolation problem at hand.}
    \label{fig:interperror}
\end{figure}

\subsection{\label{subsec:algorithm}Algorithm}
The IFGF factorization and associated box and cone interpolation
strategies and structures mentioned in the previous sections
underlie the full IFGF method---whose details are presented in what
follows. Section~\ref{sec:algnotation} introduces the box and cone structures themselves, together with the associated multi-level
field evaluation strategy. The notation and definitions are then
incorporated in a narrative description of the full IFGF algorithm
presented in Section~\ref{sec:algdescription}. A pseudo-code for the
algorithm, together with a study of the algorithmic complexity of the
proposed scheme, finally, are presented in
Section~\ref{sec:algpseudocode}.

\subsubsection{Definitions and notation\label{sec:algnotation}}
The IFGF algorithm accelerates the evaluation of the discrete
operator~\eqref{eq:field1} on the basis of a certain hierarchy
$\mathcal{B}$ of boxes (each one of which provides a partitions of the
set $\Gamma_N$ of discretization points). The box hierarchy, which
contains, say, $D$ levels, gives rise to an intimately related
hierarchy $\mathcal{C}$ of interpolation cone segments. At each level $d$
($1\leq d\leq D$), the latter $D$-level hierarchy is embodied in a
cone domain partition (cf. \eqref{eq:defconedomainproduct}) in
$(s,\theta,\varphi)$ space---each partition amounting to a set of
spherical interpolation cone segments spanning all regions of space
outside certain circumscribing spheres. The details are as follows.

The level $d$ ($1\leq d\leq D$) surface partitioning is produced on
the basis of a total of $\left( 2^{d-1}\right)^3$
Cartesian boxes (see Figure~\ref{fig:Interpolationpointvarious}). The
boxes are labeled, at each level $d$, by means of certain
level-dependent multi-indices. The hierarchy is initialized by a
single box at level $d=1$,
\begin{equation} \label{eq:deffirstbox}
    B_\mathbbm{1}^{1} \coloneqq B(x_\mathbbm{1}^1, H_1) \quad (\text{cf.} \eqref{eq:defbox}),
\end{equation}
containing $\Gamma_N$ ($B_\mathbbm{1}^{1} \supset \Gamma_N$), where
$H_1 > 0$ and $x_\mathbbm{1}^1 \in \mathbb{R}^3$ denote the side and
center of the box, respectively, and where, for the sake of
consistency in the notation, the multi-index
$\mathbbm{1} \coloneqq (1, 1, 1)^T$ is used to label the single box
that exists at level $d=1$. The box $B^{1}_\mathbbm{1}$ is then
partitioned into eight level $d = 2$ equi-sized and disjoint child
boxes $B^{2}_\mathbf{k}$ of side $H_2 = H_1/2$
($\mathbf{k} \in \{1, 2\}^3$), which are then further partitioned into
eight equi-sized disjoint child boxes $B^{3}_\mathbf{k}$ of side
$H_3 = H_2 /2$ ($\mathbf{k} \in \{1,2,3,4\}^3 = \{1,\ldots,2^2\}^3$),
etc. The eight-child box partitioning procedure is continued
iteratively for all $1\leq d\leq D$, at each stage halving the box
along each one of the three coordinate directions $(x, y, z)$, and
thus obtaining, at level $d$, a total of $2^{d-1}$ boxes along each
coordinate axes. The partitioning procedure continues until level
$d = D \in \mathbb{N}$ is reached---where $D$ is chosen in such a way
that the associated box-size $H_D$ is sufficiently small. An
illustrative two-dimensional analog of the setup for the first three
levels and associated notation is presented in Figure \ref{fig:Setup}.
				
As indicated above, the box-hierarchy $\mathcal{B}$ is accompanied by
a cone segment hierarchy $\mathcal{C}$. The hierarchy $\mathcal{C}$ is
iteratively defined starting at level $d=D$ (which corresponds to the
smallest-size boxes in the hierarchy $\mathcal{B}$) and moving
backwards towards level $d=1$. At each level $d$, the cone segment
hierarchy consists of a set of cone domains $E_\gamma^d$ which,
together with certain related concepts, are defined following upon the
discussion concerning equation~\eqref{eq:defconedomain}. Thus, using
$n_{s,d}$, $n_{C,d}$ and $2 n_{C,d}$ level-$d$ interpolation intervals
in the $s$, $\theta$ and $\varphi$ variables, respectively, the
level-$d$ cone domains
\begin{equation*}
  E_\gamma^d = E_{\gamma_1}^{s; d} \times E_{\gamma_2}^{\theta; d} \times E_{\gamma_3}^{\varphi; d} \subset [0, \sqrt{3}/3] \times [0, \pi] \times [0, 2 \pi),
\end{equation*} 
and its Cartesian components $E_{\gamma_1}^{s; d}$,
$E_{\gamma_2}^{\theta; d}$ and $E_{\gamma_3}^{\varphi; d}$ (of sizes
$\Delta_{s,d}$, $\Delta_{\theta,d}$, and $\Delta_{\varphi,d}$,
respectively) are defined following the definition of $E_\gamma$
in~\eqref{eq:defconedomainproduct} and its Cartesian components,
respectively, for $n_s = n_{s,d}$ and $n_C=n_{C,d}$, and for
$\gamma = (\gamma_1, \gamma_2, \gamma_3) \in K_C^d \coloneqq \{1,
\ldots, n_{s,d} \} \times \{1, \ldots, n_{C,d} \} \times \{1, \ldots, 2 n_{C,d} \}$. Since the parametrization $\mathbf{x}$ in
\eqref{eq:defparametrizationins} depends on the box size $H=H_d$, and
thus, on the level $d$, the following notation for the $d$-level
parametrization is used
\begin{equation} \label{eq:parametrizationleveldependent}
    {\mathbf x}^d(s, \theta, \varphi) = {\mathbf x}(\frac{\sqrt{3} H_d}{2 r}, \theta, \varphi),
\end{equation}
which coincides with the expression $\eqref{eq:defparametrizationins}$
with $H = H_d$.  Using this parametrization, the level-$d$
origin-centered cone segments are then defined by
\begin{equation}\label{eq:int-seg-non-cent}
    C_\gamma^d = \{ \mathbf{x}^d(s, \theta, \varphi) \, : \, (s, \theta, \varphi) \in E_\gamma^d\} \quad \text{for all } \gamma \in K_C^d,
\end{equation}
with the resulting cone hierarchy
\begin{equation*}
    \mathcal{C} \coloneqq \{ C_\gamma^d : 1 \leq d \leq D, \, \gamma \in K_C^d \}.
\end{equation*}
The interpolation segments $C_{\mathbf{k}; \gamma}^d$ actually
used for interpolation of fields resulting from sources contained
within an individual level-$d$ box centered at the point
$x_\mathbf{k}^d$, are given by
\begin{equation}\label{eq:int-seg-cent}
C_{\mathbf{k}; \gamma}^d \coloneqq C_\gamma^d + x_\mathbf{k}^d \quad \text{for all } \gamma \in K_C^d \text{ and } \mathbf{k} \in K^d.
\end{equation}
An illustration of a two dimensional example of the cone segments and their
naming scheme can be found in Figure \ref{fig:Cones}. 

Unlike the box partitioning process, which starts from a single box
and proceeds from one level to the next by subdividing each parent box
into $2\times 2\times 2 = 8$ child boxes (with refinement factors
equal to two in each one of the Cartesian coordinate directions,
resulting in a number $8^{d-1}$ boxes at level $d$), the cone segment
partitioning approach proceeds iteratively downward, starting from the
two $d = (D+1)$ initial cone domains
\[
  E_{(1,1,1)}^{D+1} =  [0, \sqrt{3}/3] \times [0, \pi] \times [0, \pi)\quad\mbox{and}\quad E_{(1,1,2)}^{D+1} =  [0, \sqrt{3}/3] \times [0, \pi] \times [\pi, 2\pi).
\]
(The initial cone domains are only introduced as the initiators of the
partitioning process; actual interpolations are only performed from
cone domains $E_\gamma^d$ with $D\geq d\geq 1$.) Thus, starting at
level $d=D$ and moving inductively downward to $d=1$, the cone
domains at level $d$ are obtained, from those at level $(d+1)$, by
refining each level-$(d+1)$ cone domain by level-dependent refinement
factors $a_d$, i.e. the number of cone segments in radial and angular directions from one level to the next is taken as $n_{s,d-1} = n_{s,d}/a_d$ and $n_{C,d-1} = n_{C, d} / a_d$. As discussed in what follows, the
refinement factors are taken to satisfy $a_d=1$ or
$a_d=2$ for $D\geq d\geq 2$, but the initial refinement value
$a_{D+1}$ is an
arbitrary positive integer value.

The selection of the refinement factors $a_d$ for $(D+1)\geq d\geq 2$
proceeds as follows. The initial refinement factor $a_{D+1}$ is
chosen, via simple interpolation tests, so as to ensure that the
resulting level-$D$ values $\Delta_{s,D}$, $\Delta_{\theta,D}$ and
$\Delta_{\varphi,D}$ lead to interpolation errors below the prescribed
error tolerance (cf.  Theorem~\ref{theorem:errorestimatenested}). The
selection of refinement factors $a_d$ for $d=D,D-1,\dots, 2$, in turn,
also relies on Theorem~\ref{theorem:errorestimatenested} but, in this
case, in conjunction with Theorem~\ref{theorem:Error}---as discussed
in what follows in the case $\kappa H_d>1$ and, subsequently, for
$\kappa H_d\leq 1$. In the case $\kappa H_d>1$,
Theorem~\ref{theorem:Error} bounds the $n$-th derivatives of $g_S$ by
a multiple of $(\kappa H_d)^n$. It follows that, in this case, each
increase in derivative values that arise as the box size is, say,
doubled, can be offset, per Theorem~\ref{theorem:errorestimatenested},
by a corresponding decrease of the segment lengths $\Delta_{s, d}$,
$\Delta_{\theta, d}$ and $\Delta_{\varphi, d}$ by a factor of
one-half. Under this scenario, therefore, as the box-size $\kappa H_d$
is increased by a factor of two, the corresponding parent cone segment
is partitioned into eight child cone segments (using $a_d=2$)---in
such a way that the overall error bounds obtained via a combination of
Theorems~\ref{theorem:errorestimatenested} and~\ref{theorem:Error}
remain unchanged for all levels $d$, $1\leq d\leq
D$. Theorem~\ref{theorem:Error} also tells us that, in the
complementary case $\kappa H_d\leq 1$ (and, assuming that,
additionally, $2\kappa H_d\leq 1$), for each $n$, the $n$-th order
derivatives remain uniformly bounded as the acoustical box-size
$\kappa H_d$ varies. In this case it follows from
Theorem~\ref{theorem:errorestimatenested} that, as the box size is
doubled and the level $d$ is decreased by one, the error level is
maintained (at least as long as the $(d-1)$-level box size
$\kappa H_{d-1} = 2\kappa H_d$ remains smaller than one), without any
modification of the domain lengths $\Delta_{s,d}$, $\Delta_{\theta,d}$
and $\Delta_{\varphi,d}$. In such cases we set $a_d=1$, so that the
cone domains remain unchanged as the level transitions from $d$ to
$(d-1)$, while, as before, the error level is maintained. The special
case in which $\kappa H_d <1$ but $2\kappa H_d >1$ is handled by
assigning the refinement factors $a_d =2$ as in the $\kappa H_d> 1$
case.  Once all necessary cone domains $E_\gamma^d$ ($D\geq d\geq 1$)
have been determined, the cone segments $C_{\mathbf{k}; \gamma}^d$
actually used for interpolation around a given box
$B^{d}_\mathbf{k}\in\mathcal{B}$ are obtained
via~\eqref{eq:int-seg-non-cent}-\eqref{eq:int-seg-cent}. A
two-dimensional illustration of the multi-level cone segment structure
is presented in Figure~\ref{fig:conestructure}.
\begin{figure}
    \centering
    \begin{subfigure} {0.9\textwidth}
    \centering
    \includegraphics[width=0.5\textwidth]{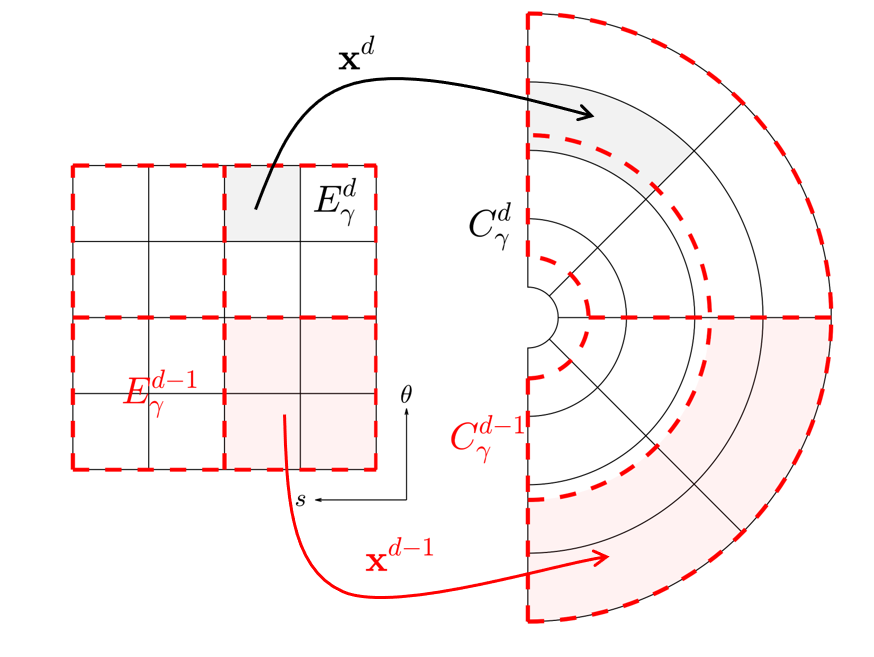}
    \caption{Two-dimensional illustration of the multi-level cone
      domains $E_\gamma^d$ and origin-centered cone segments
      $C_\gamma^d$ for two subsequent levels, shown in black and red,
      respectively.}
    \label{fig:conestructure1}
\end{subfigure}\\
\begin{subfigure} {0.9\textwidth}
    \centering
    \includegraphics[width=0.6\textwidth]{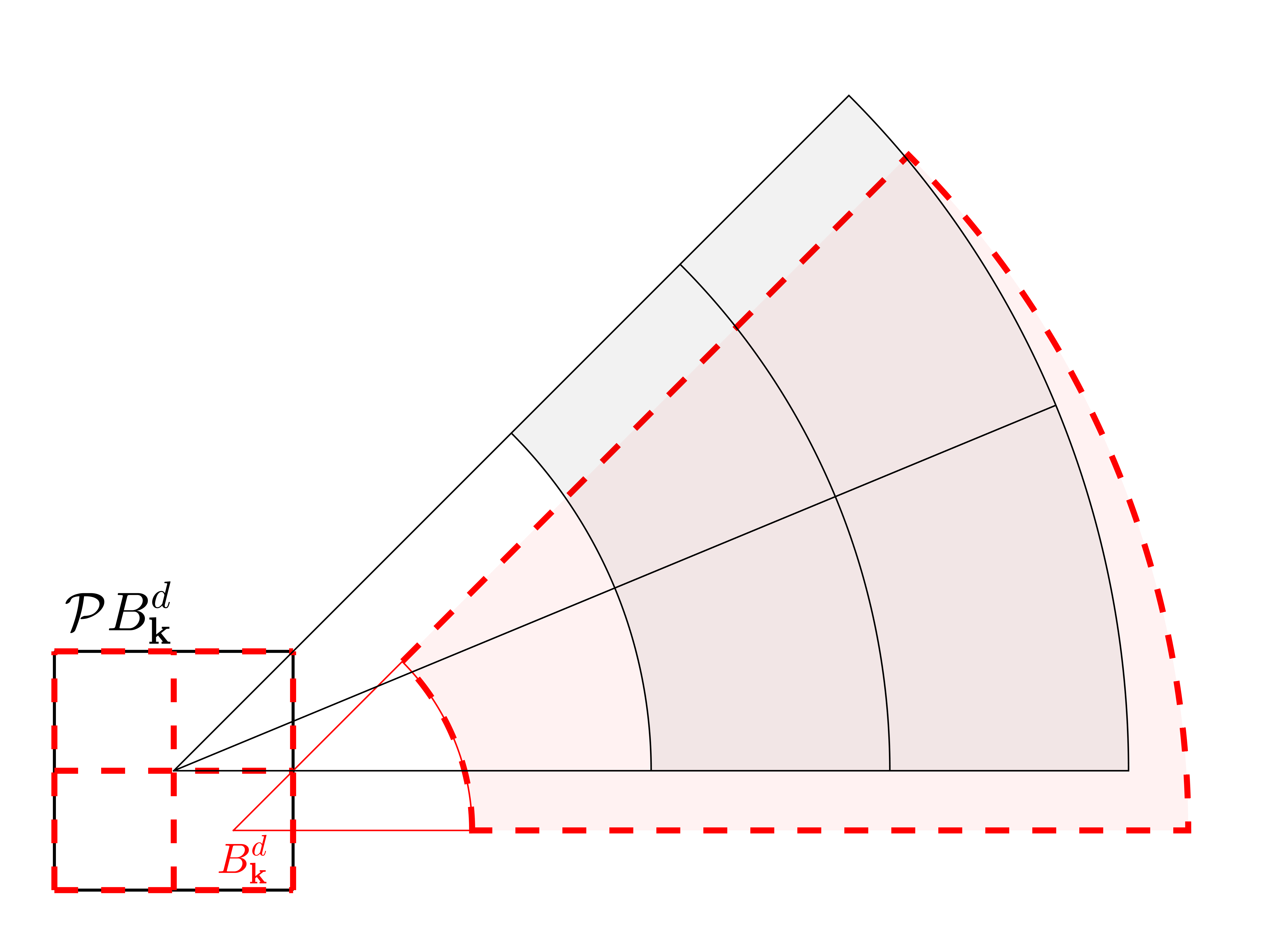}
    \caption{Two-dimensional illustration of box-centered cone
      segments, namely, a single $B_{\mathbf k}^d$-centered cone
      segment at level $d$ (in red) and the four (eight in three
      dimensions) corresponding $\mathcal{P} B_{\mathbf k}^d$-centered
      refined child cone segments at level $d-1$ depicted (in black).}
    \label{fig:conestructure2}
\end{subfigure}
\caption{Two-dimensional illustration of the hierarchical cone domain
  structure in $(s, \theta)$ space, and corresponding origin-centered
  and box-centered cone segments.}
    \label{fig:conestructure}
\end{figure}

In order to take advantage of these ideas, the IFGF algorithm
presented in subsequent sections relies on a set of concepts and
notations---including the box and cone segment structures
$\mathcal{B}$ and $\mathcal{C}$---that are introduced in what
follows. Using the notation~\eqref{eq:defbox}, the multi-index set
$K^d \coloneqq \{1, \ldots, 2^{d-1} \}^3$ (which enumerates the boxes
at level $d$, $d = 1, \ldots, D$) the initial box $B_\mathbbm{1}^{1}$
(equation~\eqref{eq:deffirstbox}), and the iteratively defined
level-$d$ box sizes and centers
\begin{equation}\label{eq:defboxsizesandcenters}
  H_d \coloneqq \frac{H_1}{2^{d-1}}, \quad x_{\mathbf k}^d \coloneqq x_\mathbbm{1}^1 -\frac{H_1}{2} \mathbbm{1} +  \frac{H_d}{2} (2 \mathbf{k} - \mathbbm{1}) \quad(\mathbf{k} \in K^d),
\end{equation}
the level-$d$ boxes and the octree $\mathcal{B}$ they bring
about are given by
    \begin{align*}
      B_{\mathbf{k}}^{d} \coloneqq B(x_{\mathbf{k}}^d, H_d)  \quad (\mathbf{k} \in K^d),\quad
      \mathcal{B} \coloneqq \{ B_{\mathbf{k}}^{d}\,  : \, d = 1, \ldots, D, \quad \mathbf{k} \in K^d \};
\end{align*}
note that, per equation \eqref{eq:defbox}, the boxes within the given
level $d$ are mutually disjoint. The field generated, as
in~\eqref{eq:definitionF}, by sources located at points within the box
$B_{\mathbf{k}}^{d}$ will be denoted by
\begin{equation}\label{eq:I_k}
    I_\mathbf{k}^d(x) \coloneqq \sum \limits_{x' \in B_\mathbf{k}^{d} \cap \Gamma_N} a(x') G(x, x') = G(x, x_\mathbf{k}^d) F_\mathbf{k}^d (x),  \qquad F_\mathbf{k}^d(x) \coloneqq \sum \limits_{x' \in B_\mathbf{k}^{d} \cap \Gamma_N} a(x') g_\mathbf{k}^d(x, x'),
\end{equation}
where $a(x')$ denotes the coefficient in sum \eqref{eq:fieldboxes}
associated with the point $x'$ and $g_\mathbf{k}^d = g_S$ the analytic
factor as in \eqref{eq:factor} centered at
$x_\mathbf{k}^d$. The octree structure
$\mathcal{B}$ coincides with the one used in Fast Multipole Methods
(FMMs)~\cite{Gumerov2004, 2003FMMLexingKernelIndependent,
  2006FMMRokhlin, 2007DirectionalFMMLexing}.

Typically only a small fraction of the the boxes on a given level $d$
intersect the discrete surface $\Gamma_N$; the set of all such {\em
  level-$d$ relevant boxes} is denoted by
\begin{equation*}
  \mathcal{R}_B^d \coloneqq \{B_\mathbf{k}^{d} \in \mathcal{B} \, : \, \mathbf{k} \in K^d,  B_\mathbf{k}^{d} \cap \Gamma_N \neq \emptyset \}.
\end{equation*}
Clearly, for each $d = 1, \ldots, D$ there is a total of
$N_B^d \coloneqq 2^{d-1}$ level-$d$ boxes in each coordinate
direction, for a total of $(N_B^d)^3$ level-$d$ boxes, out of which
only $\mathcal{O}\left((N_B^d)^2\right)$ are relevant boxes as
$d \to \infty$---a fact that plays an important role in the evaluation
of the computational cost of the IFGF method. The set
$\mathcal{N} B_\mathbf{k}^{d} \subset \mathcal{R}^d_B$ of boxes {\em
  neighboring} a given box $B_\mathbf{k}^{d}$ is defined as the set
of all relevant level-$d$ boxes $B_\mathbf{a}^{d}$ such that
$\mathbf{a}$ differs from $\mathbf{k}$, in absolute value, by an
integer not larger than one, in each one of the three coordinate
directions: $\norm{\mathbf{a} - \mathbf{k}}_\infty \leq 1$. The {\em
  neighborhood} $\mathcal{U} B_\mathbf{k}^{d} \subset \mathbb{R}^3$
of $B_\mathbf{k}^{d}$ is defined by
\begin{equation}
  \mathcal{U}B_\mathbf{k}^{d} \coloneqq \bigcup \limits_{B \in \mathcal{N}B_\mathbf{k}^{d}} B,\quad\mbox{where,}\quad     \mathcal{N}B_\mathbf{k}^{d} \coloneqq \left \{B_\mathbf{a}^{d} \in \mathcal{R}_B^d \, : \, \norm{\mathbf{a} - \mathbf{k}}_\infty \leq 1 \right\}.
\end{equation} 

An important aspect of the proposed hierarchical algorithm concerns
the application of IFGF interpolation methods to obtain field values
for groups of sources within a box $B_\mathbf{k}^{d}$ at points
farther than one box away (and thus outside the neighborhood of
$B_\mathbf{k}^{d}$, where either direct summation ($d=D$) or
interpolation from $(d+1)$-level boxes ($(D-1)\geq d\geq 1$) is
applied), but that are {\color{MyGreen}not sufficiently far} from the source box
$B_\mathbf{k}^{d}$ to be handled by the next level, $(d-1)$, in the
interpolation hierarchy, and which {\color{MyGreen}must therefore be handled} as part
of the $d$-level interpolation process. The associated {\em cousin box}
concept is defined in terms of the hierarchical parent-child
relationship in the octree $\mathcal{B}$, wherein the {\em parent box
} $\mathcal{P}B_\textbf{k}^{d} \in \mathcal{R}^{d-1}_B$ and the set
$\mathcal{Q}B_\mathbf{k}^{d} \subset \mathcal{R}_B^{d+1}$ of {\em
  child boxes } of the box $B_\textbf{k}^{d}$ are defined by
\begin{align*}
  \mathcal{P}B_\textbf{k}^{d} &\coloneqq B^{d-1}_\mathbf{a}\quad (\mathbf{a} \in K^{d-1}) \quad \text{provided} \quad  B_\mathbf{k}^{d} \subset B^{d-1}_\textbf{a}, \quad\mbox{and}\\
  \mathcal{Q} B_\mathbf{k}^{d} &\coloneqq \left \{B_\mathbf{a}^{d+1} \in \mathcal{R}_B^{d+1} \, : \, \mathcal{P}B_\mathbf{a}^{d+1} = B_\mathbf{k}^{d} \text{ } \right \}.
\end{align*}
This leads to the notion of {\em cousin boxes},
namely, non-neighboring $(d+1)$-level boxes which are nevertheless
children of neighboring $d$-level boxes. The {\em cousin boxes}
$\mathcal{M}B_\mathbf{k}^{d}$ and associated {\em cousin point sets}
$\mathcal{V}B_\mathbf{k}^{d}$ are given by
\begin{equation} \label{eq:defcousinboxes}
  \mathcal{M}B_\mathbf{k}^{d} \coloneqq \left( \mathcal{R}_B^d \setminus \mathcal{N}B_\mathbf{k}^{d} \right) \cap \mathcal{Q}\mathcal{N}\mathcal{P}B_\mathbf{k}^{d}\quad\text{and}\quad
  \mathcal{V}B_\mathbf{k}^{d} \coloneqq \bigcup \limits_{B \in \mathcal{M}B_\mathbf{k}^{d}} B.
\end{equation}
The concept of cousin boxes is illustrated in
Figure~\ref{fig:childrenofparentsneighbours} for a two-dimensional
example, wherein the cousins of the box $B_{(2, 1)}^3$ are shown in
gray. {\color{red}Note that, by definition, cousin boxes of side $H$
  are at a distance that is, say, no larger than $3H$ from each
  other. This implies that the number of cousin boxes of each box is
  bounded by a constant ($6^3 - 3^3 = 189$) independent of the level
  $d$ and the number $N$ of surface discretization points.}

A related set of concepts concerns the hierarchy of cone domains and
cone segments. As in the box hierarchy, only a small fraction of the
cone segments are {\em relevant} within the algorithm, which leads to
the following definitions of {\em cone segments
  $\mathcal{R}_C B_\mathbf{k}^{d}$ relevant for a box
  $B_\mathbf{k}^{d}$}, as well as the set $\mathcal{R}_C^d$ of all
{\em relevant cone segments at level $d$}. A level-$d$ cone segment
$C_{\mathbf{k}; \mathbf{\gamma}}^d$ is recursively defined to be
relevant to a box $B_\mathbf{k}^{d}$ if either, (i)~It includes a
surface discretization point on a cousin of $B_\mathbf{k}^{d}$, or if,
(ii)~It includes a point of a relevant cone segment associated with the
parent box $\mathcal{P}B_\textbf{k}^{d}$. In other words,
\begin{align}
  \mathcal{R}_C B_\mathbf{k}^{d} &\coloneqq \left\{ C_{\mathbf{k}; \mathbf{\gamma}}^d \, : \, \gamma \in K_C^d\, , \, C_{\mathbf{k}; \mathbf{\gamma}}^d \cap \Gamma_N \cap \mathcal{V} B_\mathbf{k}^{d} \neq \emptyset \text{ or } C_{\mathbf{k}; \mathbf{\gamma}}^d \cap \left( \bigcup \limits_{C \in \mathcal{R}_C \mathcal{P} B_\mathbf{k}^{d}}  C \right) \neq \emptyset \right \}\quad \text{and} \nonumber \\
  \mathcal{R}_C^d &\coloneqq \{ C_{\mathbf{k}; \mathbf{\gamma}}^d \in \mathcal{R}_C B_\mathbf{k}^{d} \, :\, \gamma \in K_C^d\, , \mathbf{k}\in K^d\, \mbox{and}\, B_\mathbf{k}^d \in \mathcal{R}_B^{d} \}. \label{eq:defallrelevantconesegments}
\end{align}
Clearly, whether a given cone segment is relevant to a given box on a
given level $d$ depends on the relevant cone segments on the parent
level $d-1$, so that determination of all relevant cone segments can
be achieved by means of a single sweep through the data structure,
from $d = 1$ to $d = D$.

It is important to note that, owing to the placement of the
discretization points on a two-dimensional surface $\Gamma$ in
three-dimensional space, the number of relevant boxes is reduced by a
factor of $1/4$ as the level is advanced from level $(d+1)$ to level
$d$ (at least, asymptotically as $d\to\infty$). Similarly, under the
cone segment refinement strategy proposed in view of
Theorem~$\ref{theorem:Error}$, the overall number of relevant cone
segments per box is increased by a factor of four as the box size is
doubled, so that the total number of relevant cone segments remains
essentially constant as $D$ grows:
$|\mathcal{R}_C^d| \sim |\mathcal{R}_C^{d+1}|$ for all
$d = 1, \ldots, D-1$ as $D\to\infty$, where $|\mathcal{R}_C^d|$
denotes the total number of relevant cone segments on level $d$.

As discussed in Section \ref{subsec:interpolation}, the cone segments
$C_{\mathbf{k}; \mathbf{\gamma}}^d$, which are part of the IFGF
interpolation strategy, are used to effect piece-wise Chebyshev
interpolation in the spherical coordinate system
$(s, \theta, \varphi)$. The interpolation approach, which is based on
use of discrete Chebyshev expansions, relies on use of a set
$\mathcal{X} C_{\mathbf{k}; \mathbf{\gamma}}^d$ for each relevant cone
segment $C_{\mathbf{k}; \mathbf{\gamma}}$ containing
$P = P_s \times (P_\text{ang})^2$ Chebyshev {\em interpolation  points} for all ${\mathbf k} \in K^d$ and $\gamma \in K_C^d$:
\begin{equation}\label{eq:interp_pts}
    \mathcal{X} C_{\mathbf{k}; \mathbf{\gamma}}^d = \{ x \in C_{\mathbf{k}; \mathbf{\gamma}}^d \, : \, x = {\mathbf x}^d(s_k, \theta_i, \varphi_j) + x_\mathbf{k}^d  \quad 1 \leq k \leq P_s, 1 \leq i \leq P_{\text{ang}}, 1 \leq j \leq P_{\text{ang}} \},
\end{equation}
where $s_k$, $\theta_i$ and $\varphi_j$ denote Chebyshev nodes in the
intervals $E_{\gamma_1}^{s; d}$, $E_{\gamma_2, \gamma_3}^{\theta; d}$
and $E_{\gamma_3}^{\varphi; d}$, respectively, and where
$x_\mathbf{k}^d$, which is defined
in~\eqref{eq:defboxsizesandcenters}, denotes the center of the box
$B_\mathbf{k}^d$. A two-dimensional illustration of $3 \times 3$
Chebyshev interpolation points within a single cone segment can be
found in Figure \ref{fig:Interpolationpoints}.

\begin{figure}
  \centering \subcaptionbox{A scatterer, in blue, and three levels of
    the associated box tree, with the highest level box $B^1_{(1, 1)}$
    in green, four $d=2$ level boxes in red, and sixteen $d=3$ level
    boxes, in
    black. \label{fig:Setup}}[0.485\linewidth]{\includegraphics[width=0.4\textwidth]{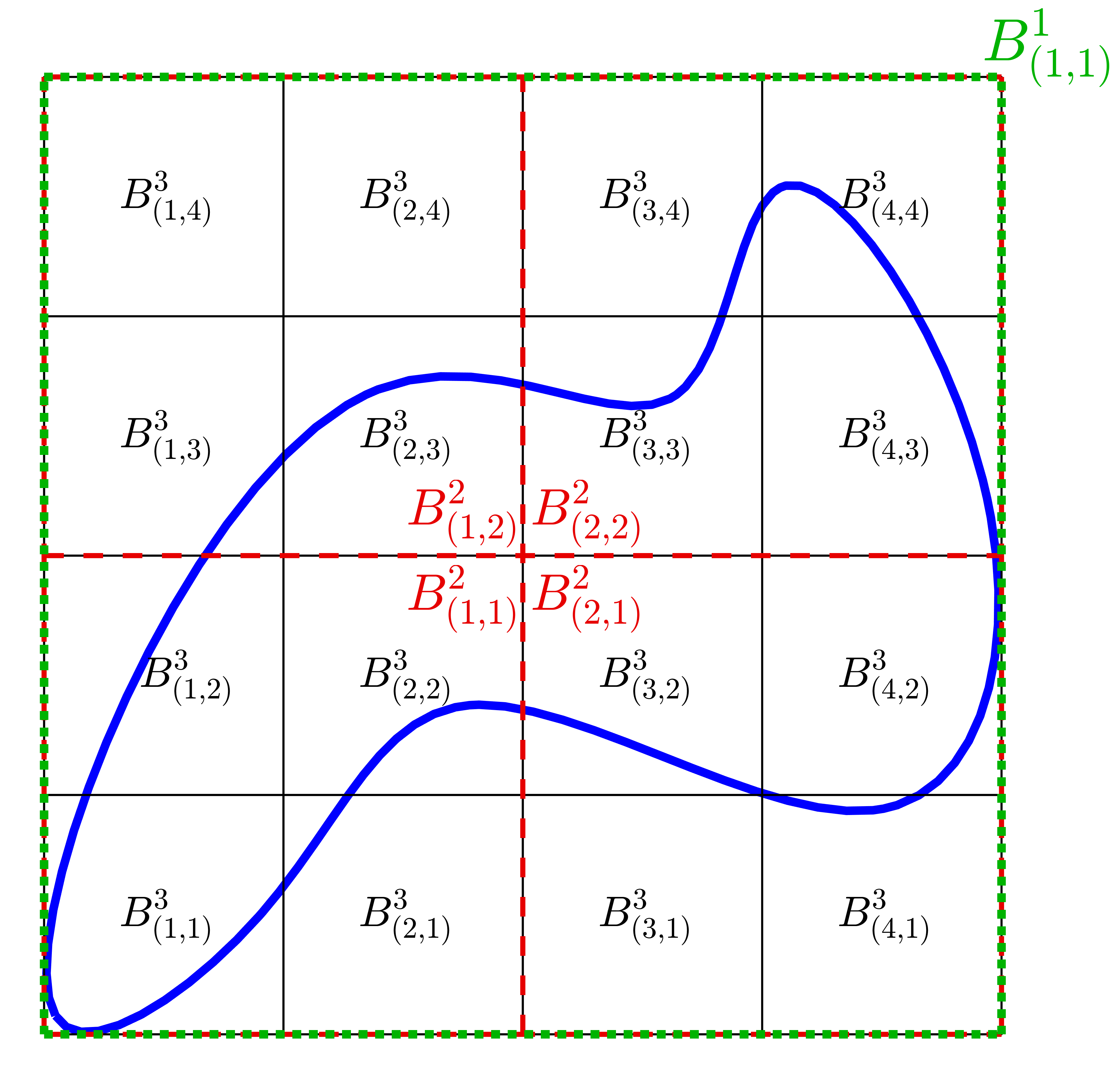}}
  \hspace{0.01\linewidth}
  \subcaptionbox{Cousins (non-neighboring children of neighbors of parents) of the box $B_{(2, 1)}^3$, in gray.  \label{fig:childrenofparentsneighbours}}[0.485\linewidth]{\includegraphics[width=0.35\textwidth]{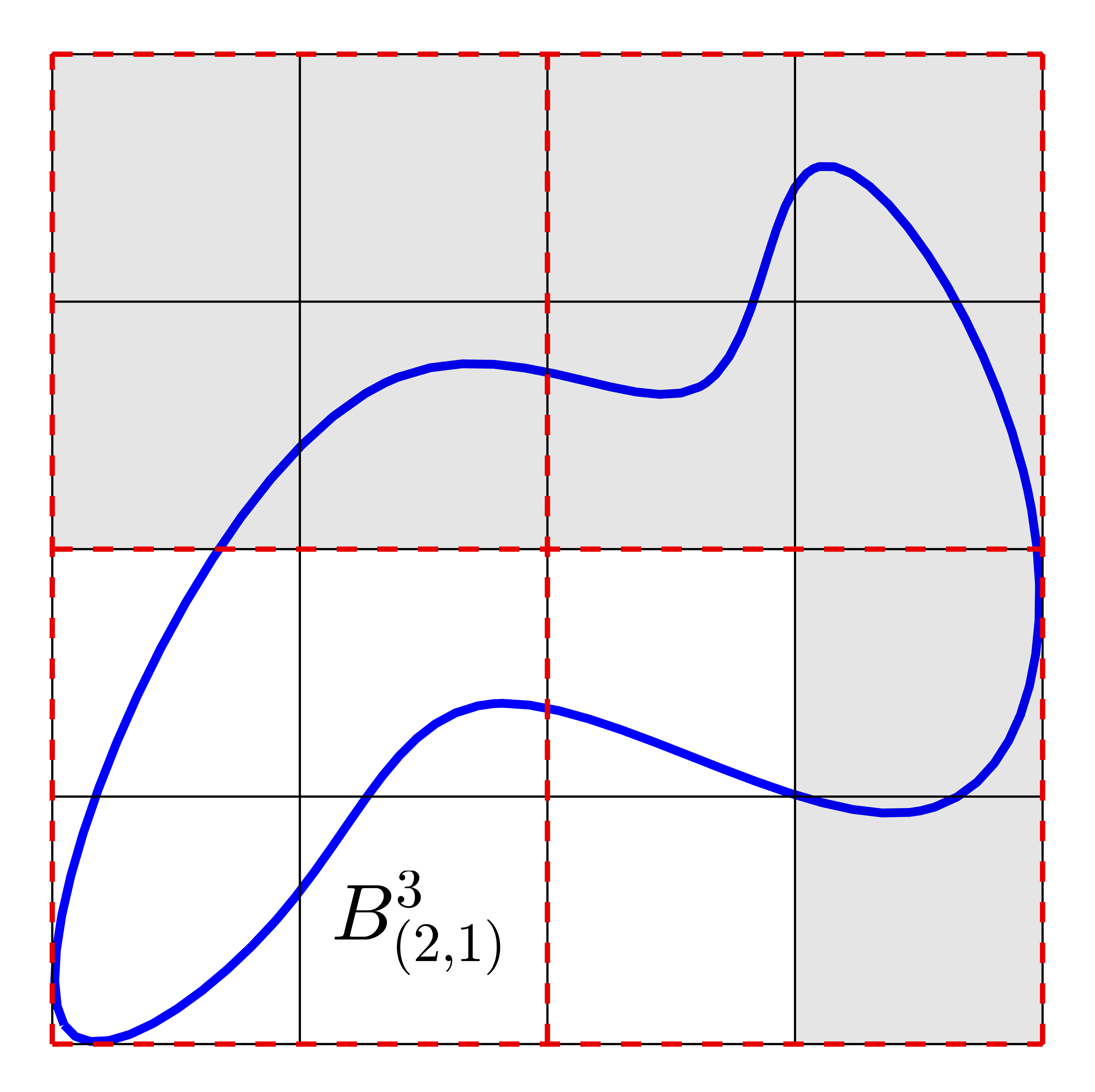}} \\
  \subcaptionbox{Illustrative sketch of the naming scheme used for
    box-centered cone segments $C_{\mathbf{k}; \mathbf{\gamma}}^d$
    (based on the level-3 box $B_{(1,
      1)}^3$). \label{fig:Cones}}[0.485\linewidth]{\includegraphics[width=0.48\textwidth]{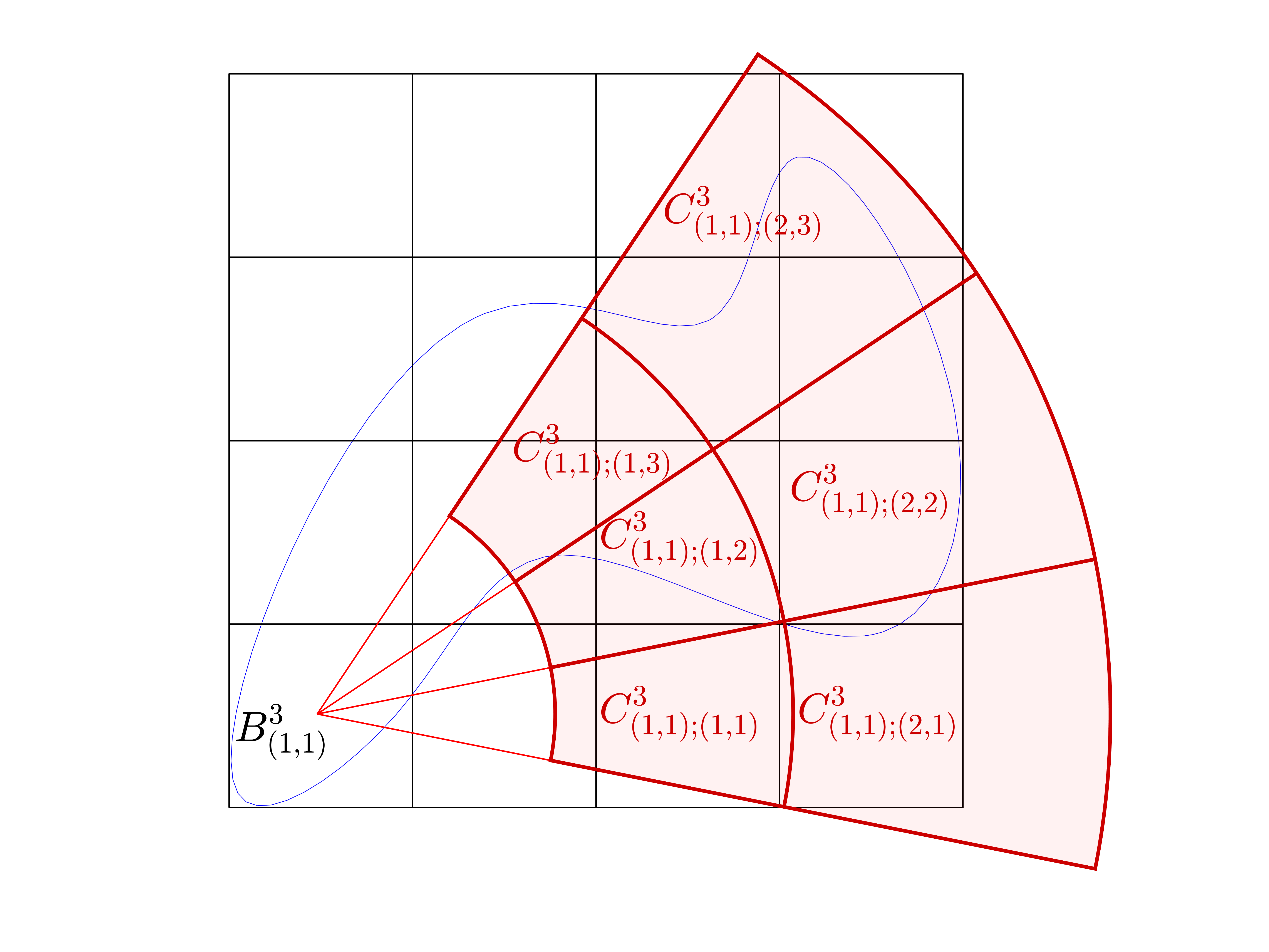}}
  \hspace{0.01\linewidth} \subcaptionbox{ $3 \times 3$ Chebyshev
    interpolation points associated with the cone segment
    $C_{(1, 1);(2,
      2)}^3$. \label{fig:Interpolationpoints}}[0.485\linewidth]{\includegraphics[width=0.48\textwidth]{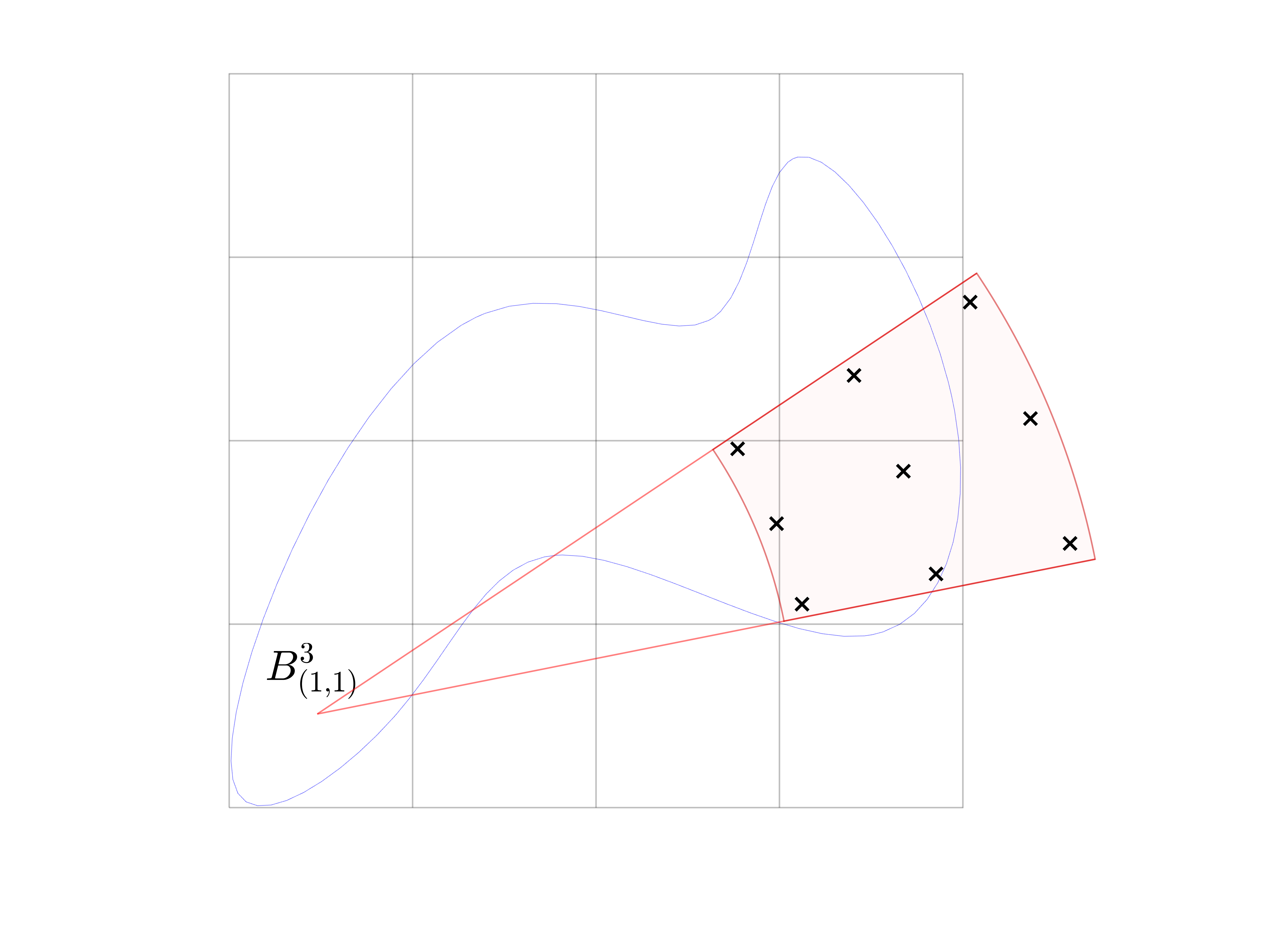}}
    \caption{Two-dimensional illustration of boxes, neighbors, cone
      segments and interpolation
      points.\label{fig:Interpolationpointvarious}}
\end{figure}

\subsubsection{Narrative Description of the
  Algorithm} \label{sec:algdescription} {\color{red} The IFGF
  algorithm consists of two main components, namely, precomputation
  and operator evaluation. The precomputation stage, which is
  performed only once prior to a series of operator evaluations (that
  may be required e.g. as part of an iterative linear-algebra solver
  for a discrete operator equation), initializes the box and cone
  structures and, in particular, it flags the relevant boxes and cone
  segments. The relevant boxes at each level $d$ ($1 \leq d \leq D$)
  are determined, at a cost of $\mathcal{O}(N)$ operations, by
  evaluation of the integer parts of the quotients of the coordinates
  of each point $x \in \Gamma_N$ by the level-$d$ box-size
  $H_d$---resulting in an overall cost of $\mathcal{O}(N\log N)$
  operations for the determination of the relevant boxes at all
  $D\sim\log N$ levels.  Turning to determination of relevant cone
  segments, we first note that, since there are no cousin boxes for
  any box in either level $d = 1$ (there is only one box in this
  level) or level $d = 2$ (all boxes are neighbours in this level), by
  definition~\eqref{eq:defallrelevantconesegments}, there are also no
  relevant cone segments in levels $d = 1$ and $d=2$. To determine the
  relevant cone segments at level $d = 3$, in turn, the algorithm
  loops over all relevant boxes $B_\mathbf{k}^3 \in \mathcal{R}_B^3$,
  and then over all cousin target points
  $x \in \Gamma_N \cap \mathcal{V}B_\mathbf{k}^3$ of $B_\mathbf{k}^3$,
  and it labels as a relevant cone segment the unique cone segment
  which contains $x$. (Noting that, per
  definition~\eqref{eq:defconesegments}, the cone segments associated
  with a given relevant box are mutually disjoint, the determination
  of the cone segment which contains the cousin point $x$ is
  accomplished at $\mathcal{O}(1)$ cost by means of simple arithmetic
  operations in spherical coordinates.) For the consecutive levels
  $d = 4, \ldots, D$, the same procedure as for level $d = 3$ is used
  to determine the relevant cone segments arising from cousin
  points. In contrast to level $d = 3$, however, for levels
  $d = 4, \ldots, D$ the relevant cone segments associated with the
  parent box $\mathcal{P} B_\mathbf{k}^d\in \mathcal{R}_B^{d-1}$ of a
  relevant box $B_\mathbf{k}^d \in \mathcal{R}_B^d$ also play a role
  in the determination of the relevant cone segments of the box
  $B_\mathbf{k}^d$. More precisely, for $d\geq 4$ the algorithm
  additionally loops over all relevant cone segments
  $C \in \mathcal{R}_C \mathcal{P} B_\mathbf{k}^d$ centered at the
  parent box and all associated interpolation points
  $x \in \mathcal{X} C$ and, as with the cousin points, flags as
  relevant the unique cone segment $C_\mathbf{k}^d$ associated with
  the box $B_\mathbf{k}^d$ that includes the interpolation point $x$.}

Once the box and cone segment structures $\mathcal{B}$ and
$\mathcal{C}$ have been initialized, and the corresponding sets of
relevant boxes and cone segments have been determined, the IFGF
algorithm proceeds to the operator evaluation stage. The algorithm
thus starts at the initial level $D$ by evaluating directly the
expression~\eqref{eq:I_k} with $d=D$ for the analytic factor
$F_\mathbf{k}^D(x)$ (which contains contributions from all point
sources contained in $B_\mathbf{k}^D$) for all level-$D$ relevant
boxes $B_\mathbf{k}^D \in \mathcal{R}_B^D$ at all the surface
discretization points $x \in \mathcal{U} B_\mathbf{k}^D \cap \Gamma_N$
neighboring $B_\mathbf{k}^D$, as well as all points $x$ in the set
$\mathcal{X} C_{\mathbf{k}; \mathbf{\gamma}}^D$ (equation
\eqref{eq:interp_pts}) of all spherical-coordinate interpolation
points associated with all relevant cone segments
$C_{\mathbf{k}; \mathbf{\gamma}}^D$ emanating from
$B_\mathbf{k}^D$. All the associated level-$D$ spherical-coordinate
interpolation polynomials are then obtained through a direct
computation of the coefficients
~\eqref{eq:defchebyshevcoeffsandpoints}, and the stage $D$ of the
algorithm is completed by using some of those interpolants to
evaluate, for all level-$D$ relevant boxes $B_\mathbf{k}^D$, the
analytic factor $F_\mathbf{k}^D(x)$ through evaluation of the sum
\eqref{eq:defchebyshevinterpooperator}, and, via multiplication by the
centered factor, the field $I_\mathbf{k}^D(x)$ at all cousin target
points $x\in \Gamma_N\cap \mathcal{V}B_\mathbf{k}^D$. (Interpolation
polynomials corresponding to regions farther away than cousins, which
are obtained as part of the process just described, are saved for use
in the subsequent levels of the algorithm.)  Note that, under the
cousin condition $x \in \Gamma_N\cap \mathcal{V} B_\mathbf{k}^D$, the
variable $s$ takes values on the compact subset $[0,\eta]$
($\eta = \sqrt{3}/3<1$) of the analyticity domain $0\leq s<1$
guaranteed by Corollary~\ref{corol:sufficientconditions}, and, thus,
the error-control estimates provided in Theorem~\ref{theorem:Error}
guarantee that the required accuracy tolerance is met at the
cousin-point interpolation step. {\color{MyGreen} Additionally, each
  cousin target point $x \in \Gamma_N\cap \mathcal{V} B_\mathbf{k}^D$
  lies within exactly one relevant cone segment
  $C_{\mathbf{k}; \gamma}^D \in \mathcal{R}_C B_{\mathbf{k}}^D$. It
  follows that the evaluation of the analytic factors~\eqref{eq:I_k}
  at a point $x$ for all source boxes $B_\mathbf{k}^D$ for which $x$
  is a level-$D$ cousin is an $\mathcal{O}(1)$ operation---since each
  surface discretization point $x \in \Gamma_N$ is a cousin point for
  no more than $189 = 6^3 - 3^3$ boxes (according to Definition
  \eqref{eq:defcousinboxes} and the explanation following
  it). Therefore, the evaluation of analytic-factor cousin-box
  contributions at all $N$ surface discretization points requires
  $\mathcal{O}(N)$ operations.} This completes the level-$D$ portion
of the IFGF algorithm.

At the completion of the level-$D$ stage the field $I_\mathbf{k}^D(x)$
generated by each relevant box $B_\mathbf{k}^D$ has been evaluated at
all neighbor and cousin surface discretization points
$x \in \Gamma_N \cap \left(\mathcal{U}B_\mathbf{k}^D \cup \mathcal{V}
  B_\mathbf{k}^D\right)$, but field values at surface points farther
away from sources,
$x \in \Gamma_N \setminus \left(\mathcal{U}B_\mathbf{k}^D \cup
  \mathcal{V}B_\mathbf{k}^D\right)$, still need to be obtained; these
are produced at stages $d = D-1, \ldots, 3$. (The evaluation process
is indeed completed at level $d=3$ since by construction we have
$\mathcal{U} B_\mathbf{k}^3 \cup \mathcal{V} B_\mathbf{k}^3 \supset
\Gamma_N$ for any $\mathbf{k}\in K^3$.) For each relevant box
$B_\mathbf{k}^d \in\mathcal{R}_B^d$, the level-$d$ algorithm
($(D-1)\geq d\geq 3$) proceeds by utilizing the previously calculated
$(d+1)$-level spherical-coordinate interpolants for each one of the
relevant children of $B_\mathbf{k}^d$, to evaluate the analytic factor
$F_\mathbf{k}^d(x)$ generated by sources contained within
$B_\mathbf{k}^d$ at all points $x$ in all the sets
$\mathcal{X} C_{\mathbf{k}; \gamma}^d$ (equation
\eqref{eq:interp_pts}) of spherical-coordinate interpolation points
associated with {\color{MyGreen} relevant} cone segments
$C_{\mathbf{k}; \gamma}^d {\color{MyGreen} \in \mathcal{R}_C
  B_\mathbf{k}^d }$ emanating from $B_\mathbf{k}^d$, which are then
used to generate the level-$d$ Chebyshev interpolants through
evaluation of the sums~\eqref{eq:defchebyshevcoeffsandpoints}. The
level-$d$ stage is then completed by using some of those interpolants
to evaluate, for all level-$d$ relevant boxes $B_\mathbf{k}^d$, the
analytic factor $F_\mathbf{k}^d(x)$ and, by multiplication with the
centered factor, the field $I_\mathbf{k}^d(x)$, at all cousin target
points $x\in \Gamma_N\cap \mathcal{V}B_\mathbf{k}^d$. {\color{MyGreen}
  As in the level $D$ case, these level-$d$ interpolations are
  performed at a cost of $\mathcal{O}(N)$ operations for all surface
  discretization points---since, as in the level-$D$ case, each
  surface discretization point (i) Is a cousin target point of
  $\mathcal{O}(1)$ boxes, and (ii) Is contained within one cone segment
  per cousin box.} This completes the algorithm. 

{\color{MyGreen} As indicated in the Introduction, the IFGF method
  does not require a downward pass through the box tree structure---of
  the kind required by FMM approaches---to evaluate the field at the
  surface discretization points. Instead, as indicated above, in the
  IFGF algorithm the surface-point evaluation is performed as part of
  a single (upward) pass throught the tree structure, with increasing
  box sizes $H_d$ and decreasing values of $d$, as the interpolating
  polynomials associated with the various relevant cone segments are
  evaluated at cousin surface points. Thus, the IFGF approach
  aggregates contributions arising from large numbers of point
  sources, but, unlike the FMM, it does so using large number of
  interpolants of a low (and fixed) degree over decreasing angular and
  radial spans, instead of using expansions of increasingly large
  order over fixed angular and radial
  spans. 
}

It is important to note that, in order to achieve the desired
acceleration, the algorithm evaluates analytic factors
$F_\mathbf{k}^d(x)$ arising from a level-$d$ box $B_\mathbf{k}^d$,
whether at interpolation points $x$ in the subsequent level, or for
cousin surface discretization points $x$, by relying on interpolation
based on (previously computed) interpolation polynomials associated
with the $(d+1)$-level relevant children boxes of $B_\mathbf{k}^d$,
instead of directly evaluating $I_\mathbf{k}^d(x)$ using
equation~\eqref{eq:I_k}. In particular, all interpolation points
within relevant cone segments on level $d$ are also targets of the
interpolation performed on level $(d+1)$. Evaluation of interpolant at
surface discretization points $x\in\Gamma_N$, on the other hand, are
restricted to cousin surface points: evaluation at all points farther
away are deferred to subsequent larger-box stages of the algorithm.

Of course, the proposed interpolation strategy requires the creation,
for each level-$d$ relevant box $B_\mathbf{k}^d$, of all level-$d$
cone segments and interpolants necessary to cover both the cousin
surface discretization points as well as all of the interpolation
points in the relevant cone segments on level $(d-1)$. We emphasize
that the interpolation onto interpolation points requires a
re-centering procedure consisting of multiplication by the level $d$
centered factors, and division by corresponding level-$(d-1)$ centered
factors (cf equation~\eqref{eq:I_k}). {\color{MyGreen} We note that,
  in particular, this re-centering procedure (whose need arises as a
  result of the algorithm's reliance on the coordinate transformation
  \eqref{eq:parametrizationleveldependent} but re-centered at the
  $d$-level cube centers for varying values of $d$) causes the set of
  the children cone segments not to be geometrically contained within
  the corresponding parent cone segment (cf. Figure
  \ref{fig:conestructure2}).} The procedure of interpolation onto
interpolation points, which is, in fact, {\color{MyGreen} an iterated}
Chebyshev interpolation method, does not result in error
amplification---as it follows from a simple variation of Theorem
\ref{theorem:errorestimatenested}.

Using the notation in Section \ref{sec:algnotation}, the IFGF
algorithm described above is summarized in its entirety in what
follows.
\begin{itemize}
\item Initialization of relevant boxes and relevant cone segments.
    \begin{itemize}
    \item Determine the sets $\mathcal{R}_B^d$ and $\mathcal{R}_C^d$
      for all $d = 1, \ldots, D$.
\end{itemize}      
    \item Direct evaluations on level $D$.
    \begin{itemize}
    \item For every $D$-level box $B_\mathbf{k}^D \in \mathcal{R}_B^D$
      evaluate the analytic factor $F_\mathbf{k}^D(x)$ generated by point
      sources within $B_\mathbf{k}^D$ at all neighboring surface
      discretization points
      $x \in \Gamma_N \cap \mathcal{U} B_\mathbf{k}^D$ by direct
      evaluation of equation~\eqref{eq:I_k}.
    \item For every $D$-level box $B_\mathbf{k}^D \in \mathcal{R}_B^D$
      evaluate the analytic factor $F_\mathbf{k}^D(x)$ at all
      interpolation points
      $x \in \mathcal{X} C_{\mathbf{k}; \mathbf{\gamma}}^D$ for all
      $C_{\mathbf{k}; \mathbf{\gamma}}^D \in \mathcal{R}_C
      B_\mathbf{k}^D$.
    \end{itemize}
  \item Interpolation, for $d = D, \ldots, 3$.
	\begin{itemize}
	\item For every every box $B_\mathbf{k}^d$ evaluate the field
          $I_\mathbf{k}^d(x)$ (equation~\eqref{eq:I_k}) at every
          surface discretization point $x$ within the cousin boxes of
          $B_\mathbf{k}^d$,
          $x \in \Gamma_N \cap \mathcal{V}B_\mathbf{k}^d$, by
          interpolation of $F_\mathbf{k}^d$ and multiplication by the
          centered factor $G(x, x_\mathbf{k}^d)$.
	\item For every every box $B_\mathbf{k}^d$ determine the
          parent box $B_\mathbf{j}^{d-1} = \mathcal{P} B_\mathbf{k}^d$
          and, by way of interpolation of the analytic factor
          $F_\mathbf{k}^d$ and re-centering by the smooth factor
          $G(x, x_\mathbf{k}^d)/ G(x, x_\mathbf{j}^{d-1})$, obtain the
          values of the parent-box analytic factors
          $F_\mathbf{j}^{d-1}$ at all level-$(d-1)$ interpolation
          points corresponding to $B_\mathbf{j}^{d-1}$---that is to
          say, at all points
          $x \in \mathcal{X} C_{\mathbf{j}; \mathbf{\gamma}}^{d-1}$
          for all
          $C_{\mathbf{j}; \mathbf{\gamma}}^{d-1} \in \mathcal{R}_C
          B_\mathbf{j}^{d-1}$ (Note: the contributions of all the
          children of $B_\mathbf{j}^{d-1}$ need to be accumulated at
          this step.)
	\end{itemize}
\end{itemize}

The corresponding pseudo code, Algorithm \ref{alg:ifgf2}, is presented
in the following section.

\subsubsection{Pseudo-code and Complexity} \label{sec:algpseudocode}
\begin{algorithm}
\begin{algorithmic}[1]
  \State \textbackslash \textbackslash Initialization.
  \For{$d = 1, \ldots, D$} \label{algstate:loopinitialization}
  \State Determine relevant boxes $\mathcal{R}^d_B$ and cone segments $\mathcal{R}_C^d$. \label{algstate:DetermineRelevantStuff} 
  \EndFor 
  \State 
  \State \textbackslash \textbackslash Direct evaluations on the lowest level. 
  \For{$B_\mathbf{k}^D \in \mathcal{R}_B^D$} \label{algstate:looprelevantboxeslevelD} 
  \For{$x \in \mathcal{U} B^D_\mathbf{k} \cap \Gamma_N$} \label{algstate:loopneighbourpointslevelD}
  \Comment{Direct evaluations onto neighboring surface points} 
  \State Evaluate $I_\mathbf{k}^D(x)$ 
  \EndFor 
  \For{$C_{\mathbf{k}; \mathbf{\gamma}}^D \in \mathcal{R}_C B_\mathbf{k}^D$} \label{algstate:looprelevantconeslevelD} \Comment{Evaluate $F$ on all relevant interpolation points} 
  \For{$x \in \mathcal{X} C_{\mathbf{k};\mathbf{\gamma}}^D$} \label{algstate:loopinterpointslevelD}
  \State Evaluate and store $F_\mathbf{k}^D(x)$. \EndFor 
  \EndFor 
  \EndFor
  \State
  \State \textbackslash \textbackslash Interpolation onto surface discretization points and parent
  interpolation points.
  \For{$d = D, \ldots, 3$}\label{algstate:loopd}
  \For{$B_\mathbf{k}^d \in \mathcal{R}_B^d$} \label{algstate:looprelevantboxes} 
  \For{$x \in \mathcal{V} B_\mathbf{k}^d \cap \Gamma_N$} \label{algstate:loopnearestneighbouringsurfacepoints}
  \Comment{Interpolate at cousin surface points} 
  \State {\color{red} Evaluate} $I_\mathbf{k}^d(x)$ by interpolation \label{algstate:interpolationtosurfacepoints}
  \EndFor 
  \If {$d > 3$}
  \Comment{Evaluate $F$ on parent interpolation points} \State Determine parent $B_\mathbf{j}^{d-1} = \mathcal{P} B_\mathbf{k}^d$
  \For{$C_{\mathbf{j}; \mathbf{\gamma}}^{d-1} \in \mathcal{R}_C B_\mathbf{j}^{d-1}$} \label{algstate:looprelevantcones} 
  \For{$x \in \mathcal{X} C_{\mathbf{j};\mathbf{\gamma}}^{d-1}$} \label{algstate:loopinterppoints}
  \State Evaluate and add $F_\mathbf{k}^d(x) G(x, x_\mathbf{k}^d)/ G(x, x_\mathbf{j}^{d-1})$
  \EndFor 
  \EndFor 
  \EndIf 
  \EndFor 
  \EndFor
  \caption{IFGF Method}
\label{alg:ifgf2}
\end{algorithmic}
\end{algorithm}

As shown in what follows, under the assumption, natural in the surface
scattering context assumed in this paper, that the wavenumber $\kappa$
does not grow faster than $\mathcal{O} (\sqrt{N})$, the IFGF
Algorithm~\ref{alg:ifgf2} runs at an asymptotic computational cost of
$\mathcal{O}(N \log N)$ operations. The complexity estimates presented
in this section incorporate the fundamental assumptions inherent
throughout this paper that fixed interpolation orders $P_s$ and
$P_\text{ang}$, and, thus, fixed numbers $P$ of interpolation points
per cone segment, are utilized.

For a given choice of interpolation orders $P_s$ and $P_\text{ang}$,
the algorithm is completely determined once the number $D$ of levels
and the numbers $n_{s, D}$ and $n_{C, D}$ of level-$D$ radial and
angular interpolation intervals are selected. For a particular
configuration, the parameters $D$, $n_{s, D}$ and $n_{C, D}$ should be
chosen in such a way that the overall computational cost is minimized
while meeting a given accuracy requirement. An increasing number $D$
of levels reduces the cost of the direct neighbour-evaluations by
performing more of them via interpolation to cousin boxes---which
increases the cost of that particular part of the algorithm. The
choice of $D$, $n_{s, D}$ and $n_{C, D}$ should therefore be such that
the overall cost of these two steps is minimized while meeting the
prescribed accuracy---thus achieving optimal runtime for the overall
IFGF method. Note that these selections imply that, for bounded values
of $n_{s, D}$ and $n_{C, D}$ (e.g., we consistently use $n_{s, D}=1$
and $n_{C, D}=2$ in all of our numerical examples) it follows that
$D=\mathcal{O}(\log N)$---since, as it can be easily checked,
e.g. increasing $N\to 4 N$ and $D\to D+1$ maintains the aforementioned
optimality of the choice of the parameter $D$. In sum, the IFGF
algorithm satisfies the following asymptotics as $\kappa \to \infty$:
$\kappa^2 = \mathcal{O}(N)$, $D = \mathcal{O}(\log{N})$,
$|\mathcal{R}_C^d| = \mathcal{O}(1)$ and
$|\mathcal{R}_B^D| = \mathcal{O}(N)$.

The complexity of the IFGF algorithm equals the number of arithmetic
operations performed in Algorithm~\ref{alg:ifgf2}. To evaluate this
complexity we first consider the cost of the level $D$ specific
evaluations performed in the ``for loop'' starting in
Line~\ref{algstate:looprelevantboxeslevelD}. This loop iterates for a
total of $\mathcal{O}(N)$ times. The inner loop starting in
Line~\ref{algstate:loopneighbourpointslevelD}, in turn, performs
$\mathcal{O}(1)$ iterations, just like the loops in the
Lines~\ref{algstate:looprelevantconeslevelD} and
\ref{algstate:loopinterpointslevelD}. In total this yields an
algorithmic complexity of $\mathcal{O}(N)$ operations.

We consider next the section of the algorithm contained in the loop
starting in Line~\ref{algstate:loopd}, which iterates
$\mathcal{O}(\log N)$ times (since $D \sim \log N$).  The loop in
Line~\ref{algstate:looprelevantboxes}, in turn, iterates
$\mathcal{O}(N/4^{D-d})$ times, since the number of relevant boxes is
asymptotically decreased by a factor of $1/4$ as the algorithm
progresses from a given level $d$ to the subsequent level
$d-1$. Similarly, the loop in
Line~\ref{algstate:loopnearestneighbouringsurfacepoints} performs
$\mathcal{O}(4^{D-d})$ iterations---since, as the algorithm progresses
from level $d$ to level $d-1$, the side $H$ of the cousin boxes
increases by a factor of two, and thus, the number of cousin discrete
surface points for each relevant box increases by a factor of
four. {\color{MyGreen} The interpolation procedure in
  Line~\ref{algstate:interpolationtosurfacepoints}, finally, is an
  $\mathcal{O}(1)$ operation since each point $x$ lies in exactly one
  cone segment associated with a given box $B_\mathbf{k}^d$
  (cf. Definition~\eqref{eq:defconesegments} of the cone segments and
  the previous discussion in Section~\ref{sec:algdescription}) and the
  interpolation therefore only requires the evaluation of a single
  fixed order Chebyshev interpolant.} A similar count {\color{MyGreen}
  as for the loop in
  Line~\ref{algstate:loopnearestneighbouringsurfacepoints}} holds for
the loop in Line~\ref{algstate:looprelevantcones} which is also run
$\mathcal{O}(4^{D-d})$ times since, going from a level $d$ to the
parent level $d-1$, the number of relevant cone segments per box
increases by a factor four. The ``for'' loop in
Line~\ref{algstate:loopinterppoints} is performed $\mathcal{O}(1)$
times since the number of interpolation points per cone segment is
constant. Altogether, this yields the desired $\mathcal{O}(N \log N)$
algorithmic complexity.

In the particular case $\kappa = 0$ the cost of the algorithm is still
$\mathcal{O}(N \log N)$ operations, in view of the
$\mathcal{O}(N \log N)$ cost required by the interpolation to surface
points.  But owing to the reduced cost of the procedure of
interpolation to parent-level interpolation points, which results as a
constant number of cone segments per box suffices for $\kappa H_d <1$
(cf. Section \ref{subsec:interpolation}), the overall $\kappa = 0$
IFGF algorithm is significantly faster than it is for cases in which
$\kappa H_d >1$ for some levels $d$. In fact, it is expected that an
algorithmic complexity of $\mathcal{O}(N)$ operations should be
achievable by a suitable modification of algorithm in the Laplace case
$\kappa = 0$, but this topic is not explored in this paper at any
length. \looseness = -1

{\color{red} Finally, we consider the algorithmic complexity of the
  pre-computation stage, namely, the loop starting in
  Line~\ref{algstate:loopinitialization}. But according to the first
  paragraph in Section~\ref{sec:algdescription}, the algorithm
  corresponding to Line~\ref{algstate:DetermineRelevantStuff} is
  executed at a computing cost of $\mathcal{O}(N)$ operations. It
  follows that the full Line~\ref{algstate:loopinitialization} loop
  runs at $\mathcal{O}(N\log N)$ operations, since
  $D = \mathcal{O}(\log N)$.}

\section{Numerical Results} \label{sec:examples} {\color{red} We
  analyze the performance of the proposed IFGF approach by considering
  the computing time and memory required by the algorithm to evaluate
  the discrete operator \eqref{eq:field1} for various $N$-point
  surface discretizations. In each case, the tests concern the
  accelerated evaluation of the full $N$-point sum~\eqref{eq:field1} at each one of $N$
  discretization points $x_\ell$, $\ell=1,\dots,N$---which, if 
  evaluated by direct addition would require a total of
  $\mathcal{O}(N^2)$ operations. We consider various configurations,
  including examples for the Helmholtz ($\kappa \neq 0$) and Laplace
  ($\kappa = 0$) Green functions, and for four different geometries,
  namely, a sphere of radius $a$, an oblate (resp. prolate) spheroid of the form
\begin{equation} \label{eq:defspheroid} \left \{ (x, y, z) \in
    \mathbb{R}^3 \, : \, \frac{x^2}{\alpha^2} + \frac{y^2}{\beta^2} +
    \frac{z^2}{\gamma^2} = a^2\right \},
\end{equation}
with  $\alpha = \beta = 1$ and $\gamma = 0.1$ (resp. $\alpha = \beta = 0.1$ and $\gamma = 1$), and the rough
radius $\approx a$ sphere defined by
\begin{equation} \label{eq:defbumpysphere}
    \left\{ x = \tilde {\mathbf x}\left (a[1+0.05\sin{(40 \theta)} \sin{(40 \varphi)}], \hspace{0.25em}\theta, \hspace{0.25em}\varphi \right) \, : \, \tilde {\mathbf x} \text{ as in } \eqref{eq:defparametrizationr},\hspace{0.25em} \theta \in [0, \pi],\hspace{0.25em} \varphi \in [0, 2 \pi) \right\}.
\end{equation}
The oblate spheroid and rough sphere are depicted in
Figures~\ref{fig:flattenedsphere} and~\ref{fig:bumpysphere},
respectively.}

\begin{figure}
\centering
    \subcaptionbox{\centering A rough sphere of radius \newline
    $r = a(1 + 0.05 \sin{(40 \theta)} \sin{(40 \varphi)})$. \label{fig:bumpysphere}}[0.4\linewidth]{\includegraphics[width=0.3\textwidth]{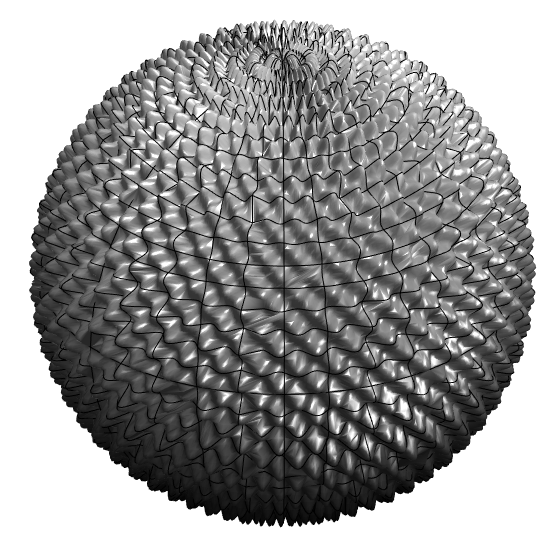}}
    \setbox9=\hbox{\includegraphics[width=.3\linewidth]{BumpySphere_cut.png}}
    \subcaptionbox{\centering An oblate spheroid given by \newline $x^2 + y^2 + (z/0.1)^2 = a^2$.\label{fig:flattenedsphere}}[0.4\linewidth]{\raisebox{\dimexpr\ht9/2-\height/2}{\includegraphics[width=0.3\textwidth]{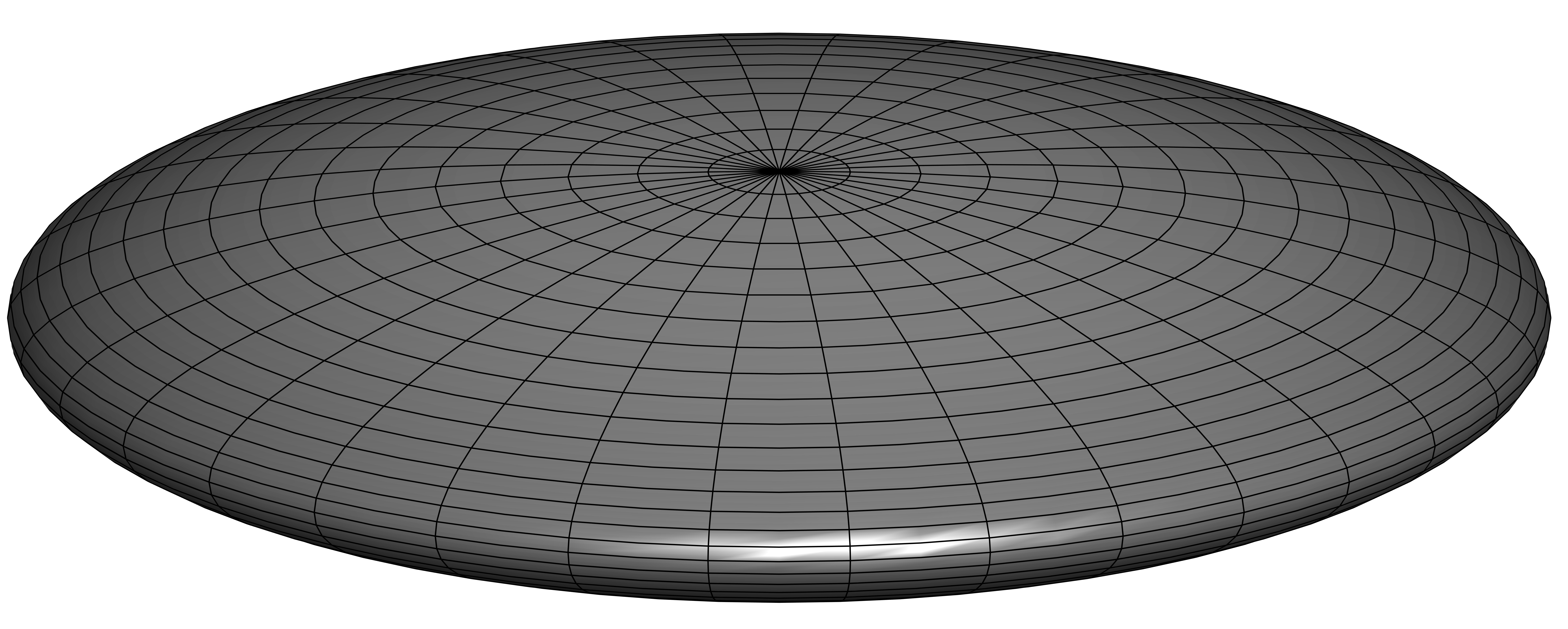}}}
    \caption{Illustration of geometries used for numerical tests.}
  \end{figure}
  All tests were performed on a Lenovo X1 Extreme 2018 Laptop with an
  Intel i7-8750H Processor and 16 GB RAM running Ubuntu 18.04 as
  operating system. The code is a single core implementation in C++ of
  Algorithm \ref{alg:ifgf2} compiled with the Intel C++ compiler
  version 19 and without noteworthy vectorization. Throughout all
  tests, {\color{red} $T_{\text{acc}}$ } denotes the time required for
  a single application of the IFGF method and excludes the
  pre-computation time $T_\text{pre}$ (which is presented separately
  in each case, and which includes the time required for setup of the
  data structures and the determination of the relevant boxes and cone
  segments), but which includes all the other parts of the algorithm
  presented in Section \ref{subsec:algorithm}, including the direct
  evaluation at the neighboring surface discretization points on level
  $D$. {\color{MyGreen} Throughout this Section, solution accuracies
    were estimated on the basis of the relative $L_2$ error norm}
\begin{equation} \label{eq:errorcomputation}
 {\color{MyGreen}  \varepsilon_M = \sqrt{\frac{\sum \limits_{i= 1}^{M} |I(x_{\sigma(i)}) - I_{\text{acc}}(x_{\sigma(i)})|^2}{\sum \limits_{i= 1}^{M} |I(x_{\sigma(i)})|^2}},}
\end{equation} 
{\color{MyGreen} on a subset of $M=1000$ points chosen randomly (using
  a random permutation $\sigma$ of the set of positive integers less
  than equal to $N$) among the $N$ surface discretization points
  $\{x_\ell:\ell = 1,\dots,N\}$ (cf. equation~\eqref{eq:field1}). Here, for a given  $x\in\Gamma_N$, $I(x)$ and $I_{\text{acc}}(x)$ denote the exact and
  accelerated evaluation, respectively, of the discrete operator
  \eqref{eq:field1} at the point $x$. To ensure that
  $M = 1000$ gives a sufficiently accurate approximation of the error,
  the exact relative errors $\varepsilon_N$ accounting for all $N$
  surface discretization points were also obtained for the first three
  test cases shown in Table~\ref{table:timings1}; the results are
  $\varepsilon_N = 3.56 \times 10^{-4}$ ($N = 24576$),
  $\varepsilon_N = 5.71 \times 10^{-4}$ ($N = 98304$) and
  $\varepsilon_N = 9.28 \times 10^{-4}$ ($N = 393216$). (Exact
  relative error evaluation for larger values of $N$ is not practical
  on account of the prohibitive computation times required by the
  non-accelerated operator evaluation.)}
The table columns display the number PPW of surface discretization
points per wavelength, the total number $N$ of surface discretization
points and the wavenumber $\kappa$. The PPW are computed
{\color{MyGreen}on the basis} of the number of surface discretization
points along the equator of a sphere (even for the rough-sphere case),
or the largest equator in the case of spheroids. Note that the PPW
have no impact on the accuracy of the IFGF acceleration, since only
the discrete operator~\eqref{eq:field1} is evaluated in the present
context, instead of an accurate approximation of a full continuous
operator. The PPW are only considered here as they provide an
indication of the discretization levels that might be used to achieve
continuous operator approximations with errors consistent with those
displayed in the various tables presented in this section.  The memory
column displays the peak memory required by the algorithm.\looseness= -1

  In all the tests where the Helmholtz Green function is used, the
  number of levels $D$ in the scheme is chosen in such a way that the
  resulting smallest boxes on level $D$ are approximately a quarter
  wavelength in size ($H_D \approx 0.25 \lambda$). Moreover, for the
  sake of simplicity, the version of the IFGF algorithm described in
  Section~\ref{subsec:algorithm} does not incorporate an adaptive box
  octree (which would stop the partitioning process once a given box
  contains a sufficiently small number of points) but instead always
  partitions boxes until the prescribed level $D$ is reached. Hence, a
  box is a leaf in the tree if and only if it is a level-$D$ box. This
  can lead to large deviations in the number of surface points within
  boxes, in the number of relevant boxes and in the number of relevant
  cone segments. These deviations may result in slight departures from
  the predicted $\mathcal{O}(N \log N)$ costs in terms of memory
  requirements and computing time. The cone segments (as defined in
  \eqref{eq:defcone}) are chosen in such a way that there are eight
  cone segments ($1 \times 2 \times 4$ segments in the $s$, $\theta$
  and $\varphi$ variables, respectively) associated with each of the
  smallest boxes on level $D$ and they are refined according to
  Section~\ref{subsec:interpolation} for the levels $d < D$. Unless
  stated otherwise, each cone segment is assigned
  $P = P_s \times P_\text{ang} \times P_\text{ang}$ interpolation
  points with $P_s = 3$ and $P_\text{ang} = 5$. {\color{red} We note
    that both, the point evaluation of Chebyshev polynomials and the
    computation of Chebyshev coefficients, are performed on the basis
    of simple evaluations of triple sums without employing any
    acceleration methods such as FFTs.}

  The first test investigates the scaling of the algorithm as the
  surface acoustic size is increased and the number of surface
  discretization points $N$ is increased proportionally to achieve a
  constant number of points per wavelength. The results of these tests
  are presented in the Tables \ref{table:timings1},
  \ref{table:timingsflattened1}, and \ref{table:timingsbumpy1} for the
  aforementioned radius-$a$ sphere, the oblate
  spheroid~\eqref{eq:defspheroid} {\color{red} (for
    $\alpha = \beta = 1$, $\gamma = 0.1$)} and rough
  sphere~\eqref{eq:defbumpysphere}, respectively. The acoustic sizes
  of the test geometries range from $4$ wavelengths to $64$
  wavelengths in diameter for the normal and rough sphere cases, and
  up to $128$ wavelengths in large diameter for the case of the oblate
  spheroid.
\begin{table} [h]
\centering
\begin{tabular}{|| r | c | c | >{\color{red}} l | l | l | r ||} 
 \hline
 \multicolumn{1}{|| c |}{\bf $N$} & \multicolumn{1}{|c|}{$\mathbf \kappa a$} & \multicolumn{1}{|c|}{\bf PPW} & \multicolumn{1}{|c|}{$\mathbf \varepsilon$} & \multicolumn{1}{|c|}{\bf $T_{\text{pre}}$ (s)} & \multicolumn{1}{|c|}{\bf {\color{red} $T_{\text{acc}}$ } (s)} & \multicolumn{1}{|c||}{\bf Memory (MB)} \Tstrut \Bstrut \\ \hline \hline
 $24576$ & $4\pi$ &\multirow{5}{*}{$22.4$}& $3.57\times 10^{-4}$ & $5.25 \times 10^{-1}$ & $1.81\times 10^{0}$ & $25$ \Tstrut\\
 $98304$ & $8\pi$ && $5.77\times 10^{-4}$ & $3.33\times 10^{0}$ & $9.30\times 10^{0}$ & $80$ \Strut\\
 $393216$ & $16\pi$ && $9.31\times 10^{-4}$ & $1.86\times 10^{1}$ & $4.55\times 10^{1}$ & $315$ \Strut \\ 
 $1572864$ & $32\pi$ && $1.49\times 10^{-3}$ & $9.74\times 10^{1}$ & $2.21\times 10^{2}$ & $1308$ \Strut \\
 $6291456$ & $64\pi$ && $1.99\times 10^{-3}$ & $4.89\times 10^{2}$ & $1.05\times 10^{3}$ & $5396$ \Bstrut \\ 
 \hline
\end{tabular}
\caption{Computing times {\color{red} $T_{\text{acc}}$ } required by the IFGF accelerator for a
  sphere of radius $a$, with $(P_s,P_\text{ang}) = (3,5)$, and for
  various numbers $N$ of surface discretization points and wavenumbers
  $\kappa a$---at a fixed number of points-per-wavelength. The
  pre-computation times $T_{\text{pre}}$, the resulting relative
  accuracy $\varepsilon$ and the peak memory used are also
  displayed. }
\label{table:timings1}
\end{table}

\begin{table} [h]
\centering
\begin{tabular}{|| r | c | c | >{\color{red}}l | l | l | r ||} 
 \hline
 \multicolumn{1}{|| c |}{\bf $N$} & \multicolumn{1}{|c|}{$\mathbf \kappa a$} & \multicolumn{1}{|c|}{\bf PPW} & \multicolumn{1}{|c|}{$\mathbf \varepsilon$} & \multicolumn{1}{|c|}{\bf $T_{\text{pre}}$ (s)} & \multicolumn{1}{|c|}{\bf {\color{red} $T_{\text{acc}}$ } (s)} & \multicolumn{1}{|c||}{\bf Memory (MB)} \Tstrut \Bstrut \\ \hline \hline
 $24576$ & $4\pi$ & \multirow{6}{*}{$22.4$} & $1.18\times 10^{-4}$ & $1.30 \times 10^{-1}$ & $1.44\times 10^{0}$ & $17$ \Tstrut\\
 $98304$ & $8\pi$ && $1.82\times 10^{-4}$ & $1.15\times 10^{0}$ & $6.52\times 10^{0}$ & $42$ \Strut\\
 $393216$ & $16\pi$ && $2.26\times 10^{-4}$ & $5.03\times 10^{0}$ & $2.87\times 10^{1}$ & $158$ \Strut \\ 
 $1572864$ & $32\pi$ && $2.55 \times 10^{-4}$ & $2.63 \times 10^{1}$ & $1.31 \times 10^{2}$ & $605$ \Strut \\
 $6291456$ & $64\pi$ && $2.83 \times 10^{-4}$ & $1.30 \times 10^{2}$ & $5.72 \times 10^{2}$ & $2273$ \Strut \\ 
  $25165824$ & $128\pi$ && $3.61 \times 10^{-4}$ & $6.27 \times 10^{2}$ & $2.64 \times 10^{3}$ & $9264$ \Bstrut \\ 
 \hline
\end{tabular}
\caption{Same as Table~\ref{table:timings1} but for an oblate spheroid
  of equation $x^2 + y^2 + (z/0.1)^2 = a^2$ depicted in
  Figure~\ref{fig:flattenedsphere}. }
\label{table:timingsflattened1}
\end{table}

\begin{table} [h]
\centering
\begin{tabular}{|| r | c | c | >{\color{red}}l | l | l | r ||} 
 \hline
 \multicolumn{1}{|| c |}{\bf $N$} & \multicolumn{1}{|c|}{$\mathbf \kappa a$} & \multicolumn{1}{|c|}{\bf PPW} & \multicolumn{1}{|c|}{$\mathbf \varepsilon$} & \multicolumn{1}{|c|}{\bf $T_{\text{pre}}$ (s)} & \multicolumn{1}{|c|}{\bf {\color{red} $T_{\text{acc}}$ } (s)} & \multicolumn{1}{|c||}{\bf Memory (MB)} \Tstrut \Bstrut \\ \hline \hline
 $24576$ & $4\pi$ & \multirow{5}{*}{$22.4$} & $2.90 \times 10^{-4}$ & $5.90 \times 10^{-1}$ & $1.90\times 10^{0}$ & $26$ \Tstrut\\
 $98304$ & $8\pi$ && $3.26\times 10^{-4}$ & $4.12 \times 10^{0}$ & $1.08\times 10^{1}$ & $97$ \Strut\\
 $393216$ & $16\pi$ && $5.08\times 10^{-4}$ & $2.58\times 10^{1}$ & $6.11\times 10^{1}$ & $463$ \Strut \\ 
 $1572864$ & $32\pi$ && $2.66 \times 10^{-4}$ & $1.47\times 10^{2}$ & $3.28 \times 10^{2}$ & $2131$ \Strut \\
 $6291456$ & $64\pi$ && $2.56\times 10^{-4}$ & $7.77\times 10^{2}$ & $1.66\times 10^{3}$ & $10502$ \Bstrut \\
 \hline
\end{tabular}
\caption{Same as Table~\ref{table:timings1} but for the rough sphere
  $r = a(1 + 0.05 \sin{(40 \theta)} \sin{(40 \varphi)})$ depicted in
  Figure~\ref{fig:bumpysphere}. }
\label{table:timingsbumpy1}
\end{table}

Several key observations may be drawn from these results. On one hand
we see that, in all cases the computing and memory costs of the method
scale like $\mathcal{O}(N \log N)$, {\color{MyGreen}thus yielding the
  expected improvement over the $\mathcal{O}(N^2)$ costs required by
  the straightforward non-accelerated algorithm}. Additionally, we
note that the computational times and memory required for a given $N$,
which are essentially proportional to the number of relevant cone
segments used, depend on the character of the surface considered
(since the number of relevant cone segments used is heavily dependent
on the surface character), and they can therefore give rise to
significant memory and computing-cost variations in some cases. For
the oblate spheroid case, for example, the number of relevant cone
segments in upward- and downward-facing cone directions is
significantly smaller than the number for the regular sphere case,
whereas the rough sphere requires significantly more relevant cone
segments than the regular sphere, especially in the $s$ variable, to
span the {\color{MyGreen} thickness} of the {\color{MyGreen} roughness
  region}.\looseness = -1

Table \ref{table:timings2} demonstrates the scaling of the IFGF method
for a fixed number $N$ of surface discretization points and increasing
wavenumber $\kappa a$ for the sphere geometry. The memory requirements
and the timings also scale like $\mathcal{O}(\kappa^2 \log \kappa)$
since the interpolation to interpolation points used in the algorithm
is independent of $N$ and scales like
$\mathcal{O}(\kappa^2 \log \kappa)$. But the time required for the
interpolation back to the surface depends on $N$ and is therefore
constant in this particular test---which explains the slight
reductions in overall computing times for a given value of $\kappa a$
over the ones displayed in Table~\ref{table:timings1} for the case in
which $N$ is scaled proportionally to $\kappa a$.

\begin{table} [h]
\centering
\begin{tabular}{|| c | c | r | >{\color{red}}l | l | l | r ||} 
 \hline
 \multicolumn{1}{|| c |}{\bf $N$} & \multicolumn{1}{|c|}{$\mathbf \kappa a$} & \multicolumn{1}{|c|}{\bf PPW} & \multicolumn{1}{|c|}{$\mathbf \varepsilon$} & \multicolumn{1}{|c|}{\bf $T_{\text{pre}}$ (s)} & \multicolumn{1}{|c|}{\bf {\color{red} $T_{\text{acc}}$ } (s)} & \multicolumn{1}{|c||}{\bf Memory (MB)} \Tstrut \Bstrut \\ \hline \hline
 \multirow{3}{*}{$393216$} & $16 \pi$ &$22.4$& $9.31\times 10^{-4}$ & $1.86\times 10^{1}$ & $4.55\times 10^{1}$ & $315$ \Tstrut \\ 
 & $32\pi$ &$11.2$& $1.13\times 10^{-3}$ & $8.17\times 10^{1}$ & $1.33\times 10^{2}$ & $1032$ \Strut\\
 & $64\pi$ &$5.6$& $1.29\times 10^{-3}$ & $3.73\times 10^{2}$ & $5.63\times 10^{2}$ & $3927$ \Bstrut\\
 \hline
\end{tabular}
\caption{Same as Table~\ref{table:timings1} but for a fixed number $N$
  of surface discretization points, demonstrating the scaling of the
  algorithm as $\kappa a$ is increased independently of the
  discretization size while maintaining the accelerator's accuracy.}
\label{table:timings2}
\end{table}

Table \ref{table:timings3} shows a similar sphere test but for a
sphere of constant acoustic size and with various numbers $N$ of
surface discretization points. As we found earlier, the computation
times and memory requirements scale like $\mathcal{O}(N \log N)$ (the
main cost of which stems from the process of interpolation back
to the surface discretization points; see Line
\ref{algstate:loopnearestneighbouringsurfacepoints} in Algorithm
\ref{alg:ifgf2}). Since the cost of the IFGF method (in terms of
computation time and memory requirements) is usually dominated by the
cost of the interpolation to interpolation points (
Line
\ref{algstate:looprelevantcones} in Algorithm \ref{alg:ifgf2}), which
is only dependent on the wavenumber $\kappa a$, the scaling in $N$ is
better than $\mathcal{O}(N\log N)$ until $N$ is sufficiently large, so
that the process of interpolation back to the surface discretization
points requires a large enough portion of the share of the overall
computing time---as observed in the fourth and fifth rows in
Table~\ref{table:timings3}.\looseness = -1
\begin{table} [h]
\centering
\begin{tabular}{|| r | c | r | >{\color{red}}l | l | l | r ||} 
 \hline
 \multicolumn{1}{|| c |}{\bf $N$} & \multicolumn{1}{|c|}{$\mathbf \kappa a$} & \multicolumn{1}{|c|}{\bf PPW} & \multicolumn{1}{|c|}{$\mathbf \varepsilon$} & \multicolumn{1}{|c|}{\bf $T_{\text{pre}}$ (s)} & \multicolumn{1}{|c|}{\bf {\color{red} $T_{\text{acc}}$ } (s)} & \multicolumn{1}{|c||}{\bf Memory (MB)} \Tstrut \Bstrut \\
 \hline\hline
 $24576$ & \multirow{4}{*}{$16\pi$} & $5.6$ & $2.90\times 10^{-4}$ & $9.30\times 10^{0}$ & $1.66\times 10^{1}$ & $228$ \Tstrut\\
 $98304$ &  & $11.2$ & $5.54\times 10^{-4}$ & $1.13\times 10^{1}$ & $2.23\times 10^{1}$ & $267$ \Strut\\
 $393216$ &  & $22.4$ & $9.09\times 10^{-4}$ & $1.40\times 10^{1}$ & $4.14\times 10^{1}$ & $320$ \Strut \\ 
 $1572864$ &  & $44.8$ & $1.04\times 10^{-3}$ & $2.34\times 10^{1}$ & $1.63\times 10^{2}$ & $498$ \Bstrut\\
 \hline
\end{tabular}
\caption{Same as Table~\ref{table:timings1} but for a fixed acoustic
  size $\kappa a$, demonstrating the scaling of the algorithm as $N$
  is increased independently of the acoustic size.}
\label{table:timings3}
\end{table}

{\color{red} Table~\ref{table:timings5} displays results produced by
  the IFGF method for the prolate spheroid~\eqref{eq:defspheroid} with
  $\alpha = \beta = 0.1$ and $\gamma = 1$ at relative error levels
  $\varepsilon \approx 10^{-2}$. The test demonstrates highly
  competitive results in terms of memory requirements and computation
  time for geometries as large as 512 wavelengths in size. The method
  exhibits similar efficiency for the sphere and the oblate spheroid
  geometries at the levels of accuracy presented in
  Table~\ref{table:timings5}.}

\begin{table} [h]
\centering
{\color{red}
\begin{tabular}{|| r | c | c | l | l | l | r ||} 
 \hline
 \multicolumn{1}{|| c |}{\bf $N$} & \multicolumn{1}{|c|}{$\mathbf \kappa a$} & \multicolumn{1}{|c|}{\bf PPW} & \multicolumn{1}{|c|}{$\mathbf \varepsilon$}  & \multicolumn{1}{|c|}{\bf $T_{\text{pre}}$ (s)} & \multicolumn{1}{|c|}{\bf {\color{red} $T_{\text{acc}}$ } (s)} & \multicolumn{1}{|c||}{\bf Memory (MB)} \Tstrut \Bstrut \\
 \hline\hline 
 $393216$ & $16\pi$ & \multirow{4}{*}{$22.4$} & $2.43\times 10^{-3}$ & $2.21 \times 10^0$ & $2.19\times 10^{1}$ & $98$ \Tstrut\\ 
 $1572864$ & $32\pi$ && $5.75\times 10^{-3}$ & $1.16 \times 10^1$ & $9.75\times 10^{1}$ & $371$ \Strut \\
  $6291456$ & $64\pi$ && $8.29\times 10^{-3}$ & $5.70 \times 10^1$ & $4.24\times 10^{2}$ & $1316$ \Strut \\
   $25165824$ & $128\pi$ && $9.84\times 10^{-3}$ & $2.72 \times 10^2$ & $1.85\times 10^{3}$ & $5317$ \Strut \\ 
   $25165824$ & $256\pi$ & $11.2$ & $1.23\times 10^{-2}$ & $3.89 \times 10^2$ & $2.05 \times 10^3$ & $5470$ \Strut \\
    $25165824$ & $512\pi$ & $5.6$ & $1.52\times 10^{-2}$ & $1.01 \times 10^3$ & $2.57\times 10^{3}$ & $10685$ \Bstrut \\ \hline
 \hline
\end{tabular} }
\caption{\color{red} Same as Table~\ref{table:timings1} but for a
  prolate spheroid of equation $(x/0.1)^2 + (y/0.1)^2 + z^2 = a^2$ and
  a target accuracy of $\varepsilon = 10^{-2}$ (cf. the second
  paragraph in the present section with regards to the selection of
  PPW in each case).}
\label{table:timings5}
\end{table}

In our final example, we consider an application of the IFGF method to
a spherical geometry for the Laplace equation. The results are shown
in Table~\ref{table:timingsLaplace}. A perfect $\mathcal{O}(N \log N)$
scaling is observed. Note that the portion of the algorithm
``interpolation to interpolation points'' (Line
\ref{algstate:looprelevantcones} in Algorithm \ref{alg:ifgf2}), which
requires a significant fraction of the computing time in the Helmholtz
case, runs at a negligible cost in the Laplace case---for which a
constant number of cone segments can be used throughout all levels, as
discussed in Section~\ref{subsec:interpolation}.

\begin{table} [h]
\centering
\begin{tabular}{|| r | >{\color{red}}l | l | r ||} 
 \hline
 \multicolumn{1}{|| c |}{\bf $N$} & \multicolumn{1}{|c|}{$\mathbf \varepsilon$} & \multicolumn{1}{|c|}{\bf {\color{red} $T_{\text{acc}}$} (s)} & \multicolumn{1}{|c||}{\bf Memory (MB)} \Tstrut \Bstrut \\
 \hline\hline 
 $24576$ & $1.51\times 10^{-5}$ & $7.81\times 10^{-1}$ & $25$ \Tstrut \\
 $98304$ & $1.38\times 10^{-5}$ & $3.62\times 10^{0}$ & $69$ \Strut \\
 $393216$ & $1.27\times 10^{-5}$ & $1.69\times 10^{1}$ & $246$ \Strut \\ 
 $1572864$ & $1.34\times 10^{-5}$ & $7.45\times 10^{1}$ & $962$ \Strut \\
 $6291456$ & $1.77\times 10^{-5}$ & $3.29\times 10^{2}$ & $3676$ \Bstrut \\
 \hline
\end{tabular}
\caption{Same as Table~\ref{table:timings1} but for the Laplace
  equation ($\kappa a = 0$). The pre-computation times (not shown) are
  negligible in this case, since the most cost-intensive part of the
  pre-computation algorithm, namely, the determination of the relevant
  cone segments, is not necessary in the present Laplace context. Per
  the IFGF Laplace algorithmic prescription, a fixed number of cone
  segments per box is used across all levels in the hierarchical data
  structure. }
\label{table:timingsLaplace}
\end{table}

Another possible optimization which was not used for the IFGF method but for the method presented in \cite{2007DirectionalFMMLexing} is the adaptivity in the box octree which prevents large deviations of surface discretization points per box and therefore increases the efficiency of the algorithm. Using an adaptive octree for the boxes would therefore lead to an improvement in the presented computation times and memory requirements.

\section{Conclusions\label{sec:conclusions}}
This paper introduced the efficient, novel and extremely simple IFGF
approach for the fast evaluation of discrete integral operators of
scattering theory. Only a serial implementation was demonstrated here
but, as suggested in the introduction, the method lends itself to
efficient parallel implementation in distributed-memory computer
clusters. Several important improvements must still be considered,
including, in addition to parallelization, adaptivity in the
box-partitioning method (so as to eliminate large deviations of
surface discretization points per box which impact negatively on the
efficiency of the algorithm) and, as suggested in the introduction,
accelerated Chebyshev interpolations of adequately higher orders while
avoiding use of large scale FFTs. Only the single layer potentials for
the Helmholtz and Laplace Green functions were considered here, but
the proposed methodology is applicable, with minimal modifications, in
a wide range of contexts, possibly including elements such as double
layer potentials, mixed formulations, electromagnetic and elastic
scattering problems, dielectric problems and Stokes flows, as well as
volumetric distribution of sources, etc. Studies of the potential
advantages offered by the IFGF strategies in these areas, together
with the aforementioned projected algorithmic improvements, are left
for future work.

\section*{Acknowledgments}
This work was supported by NSF and DARPA through contracts DMS-1714169
and HR00111720035, and the NSSEFF Vannevar Bush Fellowship under
contract number N00014-16-1-2808.

\bibliographystyle{abbrv}
\bibliography{IFGF}

\end{document}